\documentclass[oneside,a4paper,limits]{amsart}
\usepackage[left=3cm, right=2cm, top=2cm]{geometry}
\usepackage{esint}
\usepackage{amsmath}
\usepackage{amssymb}
\usepackage{amsthm}
\usepackage{amscd}
\usepackage[ansinew]{inputenc}
\usepackage{cite}
\usepackage{bbm}
\usepackage{color}
\usepackage[english=american]{csquotes}
\usepackage[final]{graphicx}
\usepackage{hyperref}
\usepackage{calc}
\usepackage{mathptmx}
\usepackage{bm}
\usepackage{enumerate}
\usepackage{amsmath,amsthm,amssymb,enumerate, fancyhdr}
\usepackage[english]{babel}
\usepackage[square,sort,comma,numbers]{natbib}

\usepackage{thmtools}

\newcommand{\curl}{{\mathrm{curl}}}
\newcommand{\setR}{\mathbb{R}}
\newcommand{\divergence}{\mathrm{div}}
\newtheorem*{theorem*}{Theorem}

\newcommand{\meantmp}[2]{#1\langle{#2}#1\rangle}
\newcommand{\mean}[1]{\meantmp{}{#1}}

\graphicspath{{../Pictures/}}

\numberwithin{equation}{section}

\newtheoremstyle{thmlemcorr}{10pt}{10pt}{\itshape}{}{\bfseries}{.}{10pt}{{\thmname{#1}\thmnumber{ #2}\thmnote{ (#3)}}}
\newtheoremstyle{thmlemcorr*}{10pt}{10pt}{\itshape}{}{\bfseries}{.}\newline{{\thmname{#1}\thmnumber{ #2}\thmnote{ (#3)}}}
\newtheoremstyle{remexample}{10pt}{10pt}{}{}{\bfseries}{.}{10pt}{{\thmname{#1}\thmnumber{ #2}\thmnote{ (#3)}}}

\theoremstyle{thmlemcorr}
\newtheorem{theorem}{Theorem}
\numberwithin{theorem}{section}
\newtheorem{lemma}[theorem]{Lemma}
\newtheorem{corollary}[theorem]{Corollary}
\newtheorem{proposition}[theorem]{Proposition}

\newtheorem{definition}[theorem]{Definition}

\theoremstyle{remexample}
\newtheorem{remark}[theorem]{Remark}

\newcommand{\Acal}{\mathcal{A}}

  \providecommand{\eta}{{\boldsymbol{\eta}}}

\newcommand{\Ibb}{\mathbb{I}}

\newcommand{\Rbb}{\mathbb{R}}

\DeclareMathOperator{\supp}{supp}

\newcommand{\norm}[1]{\|#1\|}

\newcommand{\abs}[1]{|#1|}

\newcommand{\absBB}[1]{\biggl|#1\biggr|}

\newcommand{\skp}[2]{\langle #1 , #2 \rangle}

\newcommand{\N}{\mathbb{N}}

\newcounter{assumption}
\makeatletter
\newcommand{\nextas}[1]{%
  ~\refstepcounter{assumption}%
   \protected@write \@auxout{}{\string\newlabel{#1}{{(A\theassumption)}{\thepage}{(A\theassumption)}{#1}{}}}%
   \hypertarget{#1}{(A\theassumption)}%
}
\makeatother

\def\XXint#1#2#3{{\setbox0=\hbox{$#1{#2#3}{\int}$} 
\vcenter{\hbox{$#2#3$}}\kern-.5\wd0}}

\renewcommand{\epsilon}{\varepsilon}
\renewcommand{\phi}{\varphi}

\newcommand{\Mindrila}{{M{\^i}ndril\u{a}}}
\newcommand{{\Bogovskij}}{Bogovski\u{i}}

\begin{document}


\title[]{Existence of steady very weak solutions to Navier-Stokes equations with non-Newtonian stress tensors}

\author{Claudiu \Mindrila}
\address{{\textit{C.~\Mindrila}:} Katedra matematick\'{e} anal\'{y}zy, Charles University Prague, Sokolovsk\'{a} 83, 186 75 Praha 8, Czech Republic.}
\email{mindrila@karlin.mff.cuni.cz}

\author{Sebastian Schwarzacher}
\address{\textit{S.~Schwarzacher:} Katedra matematick\'{e} anal\'{y}zy, Charles University Prague, Sokolovsk\'{a} 83, 186 75 Praha 8, Czech Republic.}
\email{schwarz@karlin.mff.cuni.cz}
%
%


\hypersetup{
  pdfauthor = {},
  pdftitle = {},
  pdfsubject = {},
  pdfkeywords = {}
}


\maketitle
\thispagestyle{empty}


\begin{abstract}
We provide existence of very weak solutions and a-priori estimates for steady flows of non-Newtonian fluids  when  the right-hand sides  are not in the natural existence class. This includes stress laws that depend non-linearly on the shear rate of the fluid like power-law fluids.
To obtain the a-priori estimates  we make use of a refined solenoidal Lipschitz truncation that preserves zero boundary values. We provide also estimates in (Muckenhoupt) weighted spaces which permit us to regain a duality pairing, which than can be used to prove existence. 
Our estimates are valid even in the presence of the convective term. They are obtained via a new comparison method that allows to ''cut out'' the singularities of the right hand side such that the skew symmetry of the convective term can be used for large parts of the right hand side. 
\vspace{4pt}



\vspace{4pt}

\noindent\textsc{Date:} \today{}. 
\end{abstract}

\setcounter{tocdepth}{2}

\section{Introduction}
\noindent
In this work we are concerned with the existence and regularity of models prescribing the motion of an incompressible non-Newtonian fluid under singular forcing. Throughout the paper we assume that $\Omega\subset\setR^3$ is a bounded Lipschitz domain and $p \in (1,\infty) $. We consider the following steady system of Navier-Stokes equations
\begin{align}
\label{eq:navier-stokes}
\begin{cases}
\mathrm{div}(u(x)\otimes u(x))-\mathrm{div}A(x,\varepsilon u(x))+\nabla\pi(x)=-\mathrm{div}\ f(x) & \mathrm{in}\ \ \Omega\\
\mathrm{div}\ u=0 & \mathrm{in}\ \ \Omega\\
u=0 & \mathrm{on}\  \partial\Omega.
\end{cases}
\end{align}
Here the unknowns are the velocity $u:\Omega\to\mathbb{R}^{3}$ and the pressure $\pi:\Omega\to\mathbb{R}$. The force is  $f:\Omega\to \mathbb{R}^{3\times 3}$  and $\varepsilon v:=\frac{1}{2}\left(\nabla v+\nabla v ^{T}\right)$ is the symmetric gradient.
The prescribed tensor $A:\Omega\times\mathbb{R}^{3\times3}\to\mathbb{R}^{3\times3}$ is a Carath\' eodory mapping; this means it is measurable in the first variable and continuous in the second variable. Additionally, we assume coercivity, boundedness and monotonicity on $A$, that is: for all $z_1 , z_2 \in \mathbb{R}^{3\times 3}$ and almost all $x\in \Omega$  the following relations hold: 
\begin{equation}
\label{eq:coerc}
A\left(x,z_{1}\right)\cdot z_{1} \ge C_1\left|z_{1}\right|^{p}-C_3,\ \mathrm{coercivity}
\end{equation}

\begin{equation}
\label{eq:bound}
\left|A\left(x,z_{1}\right)\right|\le C_2\left|z_{1}\right|^{p-1}+C_3^{\frac{p-1}{p}},\ \mathrm{boundedness}
\end{equation}
\begin{equation}
\label{eq:mon}
\left(A\left(x,z_{1}\right)-A\left(x,z_{2}\right)\right)\cdot\left(z_{1}-z_{2}\right)\ge0,\ \mathrm{monotonicity}.
\end{equation}
Observe that in case $A(x,z)\equiv \frac{\nu}{2} z$, with $\nu$ being the constant viscosity the system \eqref{eq:navier-stokes} becomes the steady  Navier Stokes equation:
\begin{align}
\label{NS}
\begin{cases}
\mathrm{div}(u(x)\otimes u(x))-\nu \Delta u  +\nabla\pi(x)=-\mathrm{div}\ f(x) & \mathrm{in}\ \ \Omega\\
\mathrm{div}\ u=0 & \mathrm{in}\ \ \Omega\\
u=0 & \mathrm{on}\  \partial\Omega.
\end{cases}
\end{align}
For the Navier-Stokes equation~\eqref{NS} in case $f\in L^q(\Omega)$ and $q\geq 2$ the existence of a solution follows by standard fixed point methods.\footnote{Please observe that we abreviate $f\in L^q(\Omega)$ for scalar functions as well as for vector valued (matrix valued) functions.} In case $q<2$ the existence can be achieved by approximating $f$ with functions in $L^2$ provided some a-priori estimates are satisfied in a space that embeds compactly in $L^2$. To the best of our knowledge the regime of exponents $q$ for which an existence theory (and respective a-priori estimates) for~\eqref{NS} are available (in three space dimensions) is $q\geq \frac{3}{2}$; see~\cite{NSA,NS1,NS2,NS3,NS3b} and the references therein. These results can not be transferred to non-linear stress tensors $A(x,\epsilon(u))$ directly (in particular not to stress tensors that depend on the solution) since they rely on the linearity of the Stokes operator. In this paper we develop an independent methodology that is suitable for non-Newtonian fluids. Our results  seem to be the first estimates in the regime of what we call \textit{very weak solutions } to equations including  non-linear stress tensors and the convective term. However, the range of admissible exponents $q$ is smaller. In case $p=2$ and under the additional hypothesis of \textit{linearity at infinity of the stress tensor} (namely~\eqref{eq:uhlenbeck}) our methods do imply the existence of solutions for exponents $q\geq \frac{12}{7}(>\frac{9}{6}=\frac{3}{2})$ (see Theorem~\ref{thm:new-main} below).

In case of $p\neq 2$ we recover the non-Newtonian fluids of Stokes type; in particular the  so-called $p$-fluids,  where $A\left(x,z\right):=\left|z\right|^{p-2}z$ ; they were  introduced by Ladyzenskaya and Lions in the late 60s ~\cite{Lady69,Lion69}.
We  point out that in this case the viscosity $\nu$ depends on the shear rate $|\varepsilon u |$ as $\nu(t) \equiv t^{p-2}$; the fluid can become shear thinning if $p<2$ or shear thickening if $p>2$. The results introduced here are new even for the following non-linear \textit{Stokes} type system:
\begin{equation}\label{eq:model}
\begin{cases}
-\mathrm{div}A(x,\varepsilon u(x))+\nabla\pi=-\mathrm{div}\ f & \mathrm{in}\ \ \Omega\\
\mathrm{div}\ u=0 & \mathrm{in}\ \ \Omega\\
u=0 & \mathrm{on}\  \partial\Omega.
\end{cases}
\end{equation}
The existence theory is motivated by the model case of $p$-fluids which are minimizers of the functional 
\[
\mathcal{F}:W_{0,\mathrm{div}}^{1,p}\left(\Omega\right)\ni v\longmapsto\int_{\Omega}\frac{\left|\varepsilon v\right|^{p}}{p}\mathrm{d}x-\int_{\Omega}f{\cdot}\nabla v\mathrm{d}x\in\mathbb{R},
\] 
 since the respective Euler Lagrange equation is
\begin{equation}\label{eq:test}
\int_{\Omega}\left|\varepsilon u\right|^{p-2}\varepsilon u{\cdot}\varepsilon\varphi\mathrm{d}x=\int_{\Omega}f{\cdot}\nabla\varphi\mathrm{d}x\ \text{for all}\ \varphi\in C_{0,\mathrm{div}}^{\infty}\left(\Omega\right).
\end{equation}
This existence approach fails in case $q<p^{\prime}$ since in this case we cannot guarantee for the coercivity of the functional. 
That relates to the fact that in this case the class of test functions has to be restricted severely. In particular we cannot use the solution $u$ as a test function if $u \in W_{0,\mathrm{div}}^{1,\tilde{q}}\left(\Omega\right)$ with $\tilde{q}<p$, only. 
Nevertheless it is possible to give the following definition of what we call a {\em very weak solution} for non-linear PDEs~\eqref{eq:navier-stokes}.
\begin{definition}\label{def:vws}
We say a function $u\in W^{1,1}_{0,\divergence}\cap L^2(\Omega)$ such that $\abs{\varepsilon u}^{p-1}\in L^1(\Omega)$ is a  \textbf{very weak solution} to \eqref{eq:navier-stokes}, if $f \in L^{q}(\Omega)$ with  $q\in [1, p^{\prime})$ and 
\begin{equation}\label{pde}
\int_{\Omega}-(u\otimes u){\cdot}\nabla \varphi+A(\cdot,\varepsilon u){\cdot}\varepsilon\varphi-\pi \divergence(\varphi)\mathrm{d}x=\int_{\Omega}f{\cdot}\nabla\varphi\mathrm{d}x\ \text{for all}\ \varphi\in C_{0}^{\infty}\left(\Omega\right).
\end{equation}
The definition for very weak solutions to~\eqref{eq:model} is analogous.
\end{definition}
\subsection*{Outline of the paper.} The paper is structured as follows: In the next section we state our main results which are the a-priori estimates for the non-Newtonian fluid laws stated above and an analytic result about solenoidal Lipschitz approximations of Sobolev functions that is necessary for the analysis  but might be of independent interest. In Section~3 we introduce some notation and preliminary results from previous papers. In Section~4 we prove the a-priori estimates and in Section~5 the existence of solutions. In the appendix the solenoidal Lipschitz approximation is constructed.

\section{Main results.}
\noindent 
The main results are a-priori estimates for systems of Stokes and Navier-Stokes type with non-Newtonian stress tensors. Up to our knowledge the paper contains the first a-priori estimates for very weak solutions of non-Newtonian fluids where the convective term is included even in the case when $p=2$. 
The a-priori estimates are given in classical Lebesgue spaces and in weighted spaces. In particular in weighted spaces where the weight {\em depends on the right hand side}. Such weighted estimates are crucial in order to apply the existence theory developed in~\cite{BuDS16} and~\cite{BBS16}. In the references it is also explained what are the advantage of the method and why other attempts are not applicable. In particular the monotonicity of the operator can {\em not be used directly}. This can be seen by observing that the (monotone) pairing $A(\cdot,\varepsilon u)\cdot \varepsilon u$ can not be expected to be an integrable function. Actually, the key observation in~\cite{BuDS16,BBS16} is that the  integrability of the monotone pair in a  certain weighted Lebesgue space can be eventually reestablished. The existence results in this work rely on~\cite[Theorem~1.9]{BBS16} (a div-curl type Lemma mentioned here as Theorem~\ref{thm:swbdcl}) which allows to reestablish the non-linearity. 

A central technical tool for the involved analysis is the so-called Lipschitz truncation method. The result we introduce in this section is a further refinement of the truncation introduced in~\cite{BDS13} and~\cite{BBS16} and it is  needed for the a-priori estimates. 
\subsection{Existence and a-priori estimates for the p-Stokes system for $p\in (1,\infty)$}
The first result we introduce implies that for $q$ close enough to  $p^{\prime}$  and $f \in L^{q}(\Omega)$  there exists a very weak solution $u\in W_{0,\mathrm{div}}^{1,q\left(p-1\right)}\left(\Omega\right)$ to \eqref{eq:model} provided that \eqref{eq:coerc}--\eqref{eq:mon} are satisfied.

  \begin{theorem}\label{thm:main2}
Let $p\in (1,\infty)$ and let  $\Omega$  be a bounded, open and Lipschitz domain. If  $A$ satisfies \eqref{eq:coerc}--\eqref{eq:mon}, then there exists an $\varepsilon_0 >0$  depending on $\Omega$, $C_1,C_2,C_3$ and $p$ such that  if $ q \in [p^{\prime}-\varepsilon_0, p^{\prime}]$ and $f\in L^{q}\left(\Omega\right)$ there exists a weak solution $\left(u,\pi\right)\in W_{0,\mathrm{div}}^{1,q(p-1)}\left(\Omega\right)\times L_{0}^{q}\left(\Omega\right)$ to~\eqref{eq:model}.
Furthermore\footnote{ For a definition of $M$ see \eqref{eq:max}.}  
\begin{align}
\label{eq:mt1}
\int_{\Omega}\left|\nabla u\right|^{q\left(p-1\right)}+\abs{\pi}^q\mathrm{d}x\le c\left(C_1,C_2,C_3,p,q,\Omega\right)\left(\int_{\Omega}\left|f\right|^{q}\mathrm{d}x+1\right)
\end{align} 
and 
\begin{align}
\label{eq:mt2}
\int_{\Omega}(\left|\nabla u\right|^{p}+\left|\pi\right|^{p'})M\left(\left|f\right|+1\right)^{q-p^{\prime}}\mathrm{d}x\le c\left(C_1,C_2,C_3,p,q,\Omega\right)\left(\int_{\Omega}\left|f\right|^{q}\mathrm{d}x+1\right)
\end{align}
for some positive constant $c\left(C_1,C_2,C_3,p,q,\Omega\right)>0$ .
\end{theorem}
As already mentioned the main novelty of this paper are the a-priori estimates, since the existence theory follows densly the approach of~\cite{BuDS16} and~\cite{BBS16}. Both estimates are new for $p\neq 2$. The respective weighted estimates for the p-Laplacian was shown in~\cite{BS16}, which also implied the existence for respective very weak solutions. The (not weighted) $L^q$ estimates for the p-Laplacian system are already known for some time, see~\cite{Iwa92,Le}. 

\subsection{Existence and a-priori estimates for the p-Navier-Stokes system $p>2$}
The second aim of this paper is to obtain a-priori estimates for the Navier-Stokes system~\eqref{eq:navier-stokes}. We wish to point out that the a-priori estimates for the p-Navier-Stokes system  do not follow in a straight forward manner from the respective estimates of the Stokes system. Unfortunately we were unable to extend the existence theory for the Navier-Stokes regime~\eqref{eq:navier-stokes} in the case when $p<2$, since the scaling of the convective term is then overwhelming the scaling of the diffusion term. 
 In case $p>2$ in case of~\eqref{eq:navier-stokes} have to replace \eqref{eq:coerc} by the following stronger assumption. We assume that for all $z_1,z_2\in \setR^{3\times 3}$
\begin{equation}
\label{eq:mons}
\left(A\left(x,z_{1}\right)-A\left(x,z_{2}\right)\right)\cdot\left(z_{1}-z_{2}\right)\ge C_1\abs{z_1-z_2}^{p}-C_3.
\end{equation}
Please observe that \eqref{eq:mons} is satisfied by the model case \eqref{eq:test} as well as by many other stress laws.
 The result for $p>2$ is the following:
\begin{theorem}\label{thm:p-NAVIER-STOKES}
Assume that $\Omega$ is a bounded Lipschitz domain. Let  $p\in (2,\infty)$ and $A$ satisfying \eqref{eq:bound}, \eqref{eq:mon} and \eqref{eq:mons}, then there exists an $\varepsilon_0 >0$  depending on $\Omega$ and $p$ such that  if $ q \in [p^{\prime}-\varepsilon_0, p^{\prime}]$ and $f\in L^{q}\left(\Omega\right)$ there exists a weak solution $\left(u,\pi\right)\in W_{0,\mathrm{div}}^{1,q(p-1)}\left(\Omega\right)\times L_{0}^{q}\left(\Omega\right)$ to~\eqref{eq:navier-stokes}. 

Furthermore,
we find
\[
\int_{\Omega}\!\!\left|\pi\right|^{q}\mathrm{d}x +
\int_{\Omega}\!\!\left|\nabla u\right|^{q\left(p-1\right)}+(\left|\nabla u\right|^{p}+\left|\pi\right|^{p'})M\left(\left|f\right|+1\right)^{q-p^{\prime}}\!\!\mathrm{d}x\le c\left(C_1,C_2,C_3,p,q,\Omega\right)\left(\int_{\Omega}\!\!\left|f\right|^{q}\mathrm{d}x+1\right)^{\frac{1}{p-2}+\alpha\big(\frac{p'-q}{q}\big)}
\]
where 
\begin{align}
\label{eq:alpha}
\alpha(s)=s \frac{2}{p-2}\text{ if } p\in (2,3)\text{ and }\alpha(s)=\max\Big\{s\frac{p}{p-2}, \frac{p-3}{p-2}\Big\}\text{ if }p>3.
\end{align}
\end{theorem}
\begin{remark}
Please observe that in case $p>3$ and $\epsilon_0$ small enough the estimate for the Navier-Stokes equation~\eqref{eq:navier-stokes} is the same as for the Stokes equation~\eqref{eq:model}. This is natural due to the scaling of the convective term which can be overwhelmed exactly when $p>3$.
\end{remark}
\subsection{Existence and a-priori estimates for the Navier-Stokes system $p=2$}
In case $p=2$ much more can be shown provided that we know for large shear speeds that the stress-tensor becomes diagonal. The additional assumption here has been introduced in~\cite{BBS16}, where the respective Stokes theory has been developed. We can extend the theory of~\cite{BBS16} to the non-linear Navier Stokes case (with convective term). We assume what we call the {\em linear at infinity condition} which says that there is a viscosity at infinity $\nu$, such that
\begin{align}
\label{eq:uhlenbeck}
\lim_{\abs{z}\to \infty} \frac{\abs{A(x,z)-\nu z}}{\abs{z}}=0\text{ and } \lim_{\abs{z}\to \infty} \abs{\partial_z A(x,z)[y]-\nu y}=0
\end{align}
uniformly in $x\in \Omega$ and $y\in \setR^{3\times 3}$.
For such stresses we have the following theorem:
  \begin{theorem}\label{thm:new-main}
Let  $\Omega$ be  a bounded, open and $C^1$  domain and $A$ satisfying \eqref{eq:coerc},\eqref{eq:bound}, \eqref{eq:mon} for $p=2$ and \eqref{eq:uhlenbeck}, then for $ q \in [\frac{12}{7}, 2]$ and $f\in L^{q}\left(\Omega\right)$ there exists a weak solution $\left(u,\pi\right)\in W_{0,\mathrm{div}}^{1,q}\left(\Omega\right)\times L_{0}^{q}\left(\Omega\right)$ to~\eqref{eq:navier-stokes}.
Furthermore  we find
\[
\int_{\Omega}\left|\pi\right|^{q}\mathrm{d}x +
\int_{\Omega}\left|\nabla u\right|^{q}+(\left|\nabla u\right|^{2}+\left|\pi\right|^{2})M\left(\left|f\right|+1\right)^{q-2}\mathrm{d}x\le C
\]
for some positive constant $C$ that depends on $\int_{\Omega}\left|f\right|^{q}\mathrm{d}x$, $C_1,C_2,C_3$ and the linear at infinity condition~\eqref{eq:uhlenbeck}.
\end{theorem}
In order to achieve the result we introduce a new decoupling method that divides the estimate by splitting the right hand side into a large part which is in the dual space (and hence the skew symmetry of the convective term may be used) and a small singular part. It is then possible to use the smallness of the mass of the singular part to quantify the difference of the Navier-Stokes solution to the Stokes solution.
\begin{remark}
In case $q\in [2,\infty)$ the existence of solutions to \eqref{eq:navier-stokes} follows by monotone operator theory and fixed point methods. The a-priori estimates (i.e.\ showing that $\nabla u\in L^q(\Omega)$ for $q>2$) then follows by \cite[Theorem 1.4]{BBS16} using $\divergence(u\otimes u)$ as part of the right hand side. 
\end{remark}

\subsection{Solenoidal Lipschitz truncations with zero boundary valules}
The last main result we present is a refinement of the Lipschitz approximation method.
Let us say a few words about the development of this tool.
Suppose  we are given a Sobolev function $u \in W^{1,1}(\Omega)$ where $\Omega$ is an open set of $\mathbb{R}^{3}$. A \textbf{Lipschitz truncation} of $u$ is a function $u_{\lambda}$  that is Lipschitz continuous with Lipschitz constant bounded by $\lambda>0$ such that $\left|\left\{ u_{\lambda}\neq u\right\} \right|\to0$ as $\lambda \to \infty $. This is done by modifying the function $u$ on the level set where  the Hardy- Lilttlewood maximal function of $\nabla u$ is greater than $\lambda$. To our knowledge, this was first achieved by Acerbi and Fusco in \cite{AF1} , \cite{AF2} and \cite{AF3}. 

The Lipschitz truncation method was  successfully applied in many areas of analysis such as: 
\begin{itemize}
\item Calculus of variations: weak lower semicontinuity for Lipschitz functions imply weak lower semicontinuity for Sobolev functions ~\cite{AF1}, \cite{AF2}, \cite{AF3}.
\item Fluid dynamics: existence of non-Newtonian fluids \cite{fms97} , \cite{dms08}, \cite{BDS13} and the references therein.
\item Very weak solutions: a-priori estimates for p-Laplacian \cite{Le} , existence and uniqueness issues \cite{BuDS16} , \cite{BDS13} \cite{BS16},  non-linear flows \cite{BBS16}. See also the recent parabolic results~\cite{DieSchStrVer17,BBS19}

\end{itemize}
A self- contained survey on Lipschitz truncations with applications to fluid dynamics  and  some  more references can also be found in the recent book \cite{breit17}.

The main ingredient for the a-priori estimates (in weighted spaces) is the use of a divergence free {\em truncation that is chosen relative to the weight}. The technique is closely related to the so-called {\em Lipschitz truncation method}.

For that we introduce the following new refinement of the solenoidal Lipschitz truncation method that was first introduced in~\cite{BDS13}.
\begin{theorem}\label{truncation} Let $p\in (1,\infty)$, let $\Omega \subset \mathbb{R}^{3}$ be an open,  bounded subset with Lipschitz boundary and let $u \in  W_{0, \text{div}}^{1,p}(\Omega)$.
Then there exists a set $\mathcal{O}\subset \Omega$, with 
\[
\abs{\mathcal{O}}\lesssim \lambda^{-p}\int_{\Omega}\abs{\nabla u}^p\, \mathrm{d}x
\] 
and a function $u_{\lambda}\in W_{0,\text{div}}^{1,\infty}\left(\Omega\right)$, such that $u(x)=u_\lambda(x)$ for all $x\in \mathcal{O}^c$.
Additionally
 $$\norm{\nabla u_{\lambda}}_{L^\infty(\Omega)}\le c\lambda
 $$
  almost everywhere, 
\begin{align*}
\int_{\mathcal{O}}\left|\nabla\left(u-u_{\lambda}\right)\right|^{q}\mathrm{d}x
&\lesssim \lambda^{q-p}\int_{\Omega} \abs{\nabla u}^p\, \mathrm{d} x
\intertext{ and }
\left\Vert u_{\lambda}-u\right\Vert _{L^{q}\left(\mathcal{O}_{\lambda}\cap\Omega\right)}^{q} 
&\lesssim\left\Vert u\right\Vert _{L^{q}(\Omega)}^{q}
\end{align*}
for all $q\in [1,p]$. The constants depend on $p$ and the domain only. Moreover, if $\nabla u\in L^p_\omega(\Omega)$ with $\omega \in A_p$ we find
\begin{align}
\int_{\mathcal{O}}\left|\nabla\left(u-u_{\lambda}\right)\right|^{p}\omega\mathrm{d}x\lesssim c\int_{\Omega}\left|\nabla\left(u-u_{\lambda}\right)\right|^{p}\omega\mathrm{d}x
\end{align}
with $c$ depending on the $A_p$ and the domain only.
\end{theorem}

\begin{remark} 
The construction of the Lipschitz truncation can be found in the Appendix (see Subsection~\ref{proof:LT}). It might be noteworthy that the proof uses the inverse-curl operator (in domains with zero traces) for which we provide  weighted estimates that might be of independent interest (see Theorem~\ref{thm:weight-curl}).

 We point out that, in contrast to earlier versions of the Lipschitz truncation,  our truncation  inherits {\bf both} the {\bf solenoidality} and the {\bf zero trace} property of the Sobolev function. In addition to the usual $L^{q}$ estimates we provide weighted $L^p$ estimates as well. We mention that our result uses the techniques recently introduced in \cite{BDS13}, \cite{BBS16}, \cite{BuDS16}, \cite{BS16}. To be more precise, we improve the inverse curl operator introduced in \cite{BS90} by proving  additional weighted estimates. 

We think this improved weighted inverse curl might be of interest as well; the complete formulation is contained in Theorem~ \ref{thm:weight-curl}.
\end{remark}

\section{Notations and preliminary results}
\subsection{Notation} In the present work  we use the following notations:  

\begin{enumerate}

\item If $E \subset \mathbb{R}^{3} $ then $\chi_{E} $ denotes the characteristic function of $E$ that assigns $1$ to each element of $E$ and otherwise is $0$; 
\item If $E$ is Lebesgue measurable we denote by $\abs E$ its Lebesgue measure ;
\item for a measurable function $f : \Omega \subseteq \mathbb{R}^{3} \to \mathbb{R}_{+}$  and a measurable set $\Omega$:
 $\int_{\Omega} f(x)\mathrm{d}x$ is the integral with respect to the Lebesgue measure and $$\frac{1}{\left|\Omega\right|}\int_{\Omega}f\left(x\right)\mathrm{d}x=:\fint_{\Omega}f\left(x\right)\mathrm{d}x=:\mean{f}_{\Omega}.$$
 \item For a function $u : \Omega  \mapsto \mathbb{R}^{3}$ we define its symmetric gradient by \[
\varepsilon u:=\frac{\nabla u+\left(\nabla u\right)^{T}}{2}. 
\]
\item Throughout the paper we usually use the letter $B$ for a ball and $Q$   for a cube with sides parallel to the axis.
\end{enumerate}

{\bf Maximal function.} The Hardy-Littlewood maximal operator is defined by 
\begin{align}
\label{eq:max}
L_{\mathrm{loc}}^{1}\left(\Omega\right)\ni f\mapsto Mf\left(x\right):=\sup_{B\ni x}\fint_{B}\left|f\left(y\right)\right|\chi_{\Omega}\left(y\right)\mathrm{d}y\in\left[0,\ \infty\right]
\end{align}
 where the supremum is considered over all the open balls that contatin $x$. This definition is extended for vector-valued functions $v\in L_{\mathrm{loc}}^{1}\left(\Omega;\mathbb{R}^{n}\right)$ by $M\left(v\right)\left(x\right):=M\left(\left|v\right|\right)\left(x\right)$.

{\bf Weights and weighted spaces.} The following notions that involve weights and weighted spaces  are well known and  we closely follow their exposure from \cite[Section 3]{BuDS16}. \\
 A function $\omega : \mathbb{R}^{n} \to \mathbb{R}$ is called a \textit{weight} if it is measurable, positive and finite almost everywhere. Given a weight $\omega$ we can define the space 
$$L_{\omega}^{p}\left(\Omega\right):=\left\{ u:\Omega\to\mathbb{R}^{n};\ \left\Vert f\right\Vert _{L_{\omega}^{p}}:=\left(\int_{\Omega}\left|u(x)\right|^{p}\omega\left(x\right)\mathrm{d}x\right)^{1/p}<\infty\right\} $$ with $1 \le p < \infty$. 
Similarly we can define the following  weighted Sobolev space :
\[
W_{\omega}^{k,q}\left(\Omega\right):=\left\{ u\in L_{\omega}^{q}\left(\Omega\right);\ \left\Vert u\right\Vert _{W_{\omega}^{k,q}\left(\Omega\right)}:=\sum_{l=0}^k\norm{\nabla^l u}_{L_{\omega}^{q}\left(\Omega\right)}<\infty\right\} 
\]
As $W^{k,q}_{0,\omega}(\Omega)$ we denote the closure of $C^\infty_0(\Omega)$ with respect to the respective weighted Sobolev norm.

{\bf Muckenhoupt weights.} We say that a weight belongs to the Muckenhoupt class $\mathcal{A}_{p}$ if and  only if
for every ball $B \subset \mathbb{R}^{n}$ we have that $$\left(\fint_{B}\omega\mathrm{d}x\right)\left(\fint_{B}\omega^{-\left(p^{\prime}-1\right)}\mathrm{d}x\right)^{1/\left(p^{\prime}-1\right)}\le A\quad\mathrm{if}\ p\in\left(1,\infty\right)$$ or $M\omega(x) \le A\omega(x)$ if $p=1$.  The smallest constant $A$ for which these inequalities hold is called the {\em Muckehoupt constant} and is denoted by $\mathcal{A}_{p}(\omega)$. One of the special features of these weights is contained in the seminal result due to B. Muckenhoup~\cite{Mu72}: if $1<p<\infty$ we have that  $\omega \in \mathcal{A}_{p}$  if and only if there exists a constant $A^{\prime}$ such that for any $f\in L^{p} (\mathbb{R})$ it follows that 
\begin{align}
\label{eq:Muc}
\int\left|Mf\right|^{p}\omega\mathrm{d}x\le A^{\prime}\int\left|f\right|^{p}\omega\mathrm{d}x.
\end{align}
\subsection{Preliminary results} 
We state below estimates and results that are further needed in the proofs of our results. 

{\bf Young's inequality with $\varepsilon$.} The following (elementary) inequality will be  used intensively: 
$$ab\le\varepsilon\frac{a^{p}}{p}+C\left(\varepsilon\right)\frac{b^{p^{\prime}}}{p^{\prime}}\ \text{ for all } a,\ b\ge0,\ \varepsilon>0,\ p>1$$ 
where $p^{\prime}$ denotes the H\"older exponent associated to $p$, that is $\frac{1}{p^{\prime}}:=1-\frac{1}{p}$ and  $C(\varepsilon)$ is a positive constant depending of $\varepsilon$. 

{\bf Estimates for Muckenhoupt weights.}
 The following two lemmas contain useful properties for the Muckenhoupt weights that we will also need.
\begin{lemma}\label{weight1} \cite[p. 5-6;]{Turesson 2000}
\label{lem:tur1}
 Let $\omega \in \mathcal{A}_{p}$ for some $p\in [1, \infty)$. Then $\omega \in \mathcal{A}_{q} $ for all $q\ge p$. Also $\omega \in \mathcal{A}_{p}$ is equivalent to $\omega^{-(p^{\prime}-1)} \in \mathcal{A}_{p^{\prime}}$.
\end{lemma}
\begin{lemma}\label{weight2} \cite[p. 5-6;]{Turesson 2000}
\label{lem:tur2}
  Let $f\in L^{1}_{\mathrm{loc}}(\mathbb{R}^{n})$ such that $Mf< \infty$ almost everywhere in $\mathbb{R}^{n}$. Then for all $\alpha \in (0,1)$ we have that $(Mf)^{\alpha} \in \mathcal{A}_{1}$. Furthermore, for all $\alpha \in (0,1)$, we have $(Mf)^{-\alpha(p-1)} \in \mathcal{A}_{p}$.
\end{lemma}
{\bf Korn and Poincar\'e 's inequalities.}
\begin{theorem}[Korn, Poincar\'{e}]
 Let  $q \in (1, \ \infty)$ and $\Omega$ be a bounded Lipschitz domain and $u\in W^{1,p}_{0}\left(\Omega;\ \mathbb{R}^{d}\right)$ . We have \[\left\Vert u\right\Vert _{W^{1,p}\left(\Omega\right)}\le c_{1}\left\Vert \nabla u\right\Vert _{L^{p}\left(\Omega\right)}\le c_{2}\left\Vert \varepsilon u\right\Vert _{L^{p}\left(\Omega\right)}\]
where $c_1$ and $c_2$ depend only on $\Omega$ and $q$.
\end{theorem}



\begin{theorem} [Muckenhoupt, Wheeden]\label{thm:weight-embedding}
\label{thm:Tur}
If $1<p<3$ we define the \textit{Sobolev exponent (in dimension $3$)} by  $p^{*}:=\frac{3p}{3-p}$. Suppose $\omega\in A_{p^{*}/p^{\prime}+1}$. Then if $ u\in W_{0}^{1,1}\left(\Omega\right)$ with $\nabla u\in L_{\omega^{\left(3-p\right)/3}}^{p}\left(\Omega\right)$ we have 
\[\left(\int_{\Omega}\left|u\right|^{p^{*}}\omega \text{d}x\right)^{1/p^{*}}\le c\left(p,A_{p^{*}/p^{\prime}+1}\left(\omega\right)\right)\left(\int_{\Omega}\left|\nabla u\right|^{p}\omega^{\left(3-p\right)/3}\text{d}x\right)^{1/p}\]
\end{theorem}

Moreover, we find by~\cite[Theorem~5.1]{DRS10} for $\Omega$ Lipschitz and $\nabla u\in L^p_\omega(\Omega)$, that
\begin{align}
\label{eq:weightponc}
\int_{\Omega}\left|u-\mean{u}_\Omega\right|^{p}\omega \text{d}x\le c\left(p,A_{p}\left(\omega\right)\right)\int_{\Omega}\left|\nabla u\right|^{p}\omega\text{d}x.
\end{align}

{\bf Weighted Korn inequality.} We record here the following weighted version of Korn's inequality . It appears in  \cite[Theorem 5.15]{DRS10} . 
\begin{theorem}
\label{thm:kornweight} Let $\Omega \subset \mathbb{R}^{d}$ be a bounded domain  , $q \in (1, \infty)$ and $\omega \in \mathcal{A}_{q}$. Then for all $u\in W_{\omega,0}^{1,q}\left(\Omega\right)$ the following inequality holds: \[
\left\Vert \nabla u\right\Vert _{L_{\omega}^{q}\left(\Omega\right)}\le c\left\Vert \varepsilon u\right\Vert _{L_{\omega}^{q}\left(\Omega\right)}
\]
 where $c$ only depends on $q$ and $A_{q}(\omega)$.

\end{theorem}

{\bf The Bogovski operator.}
 The following theorem will be essential for proving the existence of the pressure, in Theorem \ref{thm:main2}. See \cite{bog80} and~\cite[Theorem~5.2]{DRS10}.
\begin{theorem}[Bogovski] 
 Let $q \in (1, \infty)$ and $\Omega$ be a bounded Lipschitz domain and $\omega\in \mathcal{A}_q$. Then there exists a bounded linear operator 
\[
\mathrm{Bog}:L_{0,\omega}^{q}\left(\Omega\right)\mapsto W_{0,\omega}^{1,q}\left(\Omega\right)
\]
such that the system  \[
\begin{cases}
\mathrm{div}\mathrm{Bog}\left(g\right)=g & \mathrm{in}\ \Omega\\
g=0 & \mathrm{on}\ \partial\Omega
\end{cases}
\] has a weak solution and  for which we also have 
\[
\left\Vert \mathrm{Bog}\left(g\right)\right\Vert _{W^{1,q}_{\omega}\left(\Omega\right)}\le c\left\Vert g\right\Vert _{L^{q}_{\omega}\left(\Omega\right)}
\]
for some positive constant $c$ that only depends on $p$ and $\Omega$.  Here \[
L_{0,\omega}^{q}\left(\Omega\right):=\left\{ f\in L^{q}_\omega\left(\Omega\right):\ \int_{\Omega}f\, \text{d}x=0\right\} 
.\]
\end{theorem}
{\bf Solenoidal, weighted, biting div--curl lemma.}
Finally, for the existence of solutions we need the following result that can be found in~\cite[Theorem~1.9]{BBS16}.

\begin{theorem}[solenoidal, weighted, biting div--curl lemma]
\label{thm:swbdcl}
  Let $\Omega\subset \setR^n$ denote an open, bounded set. Assume that for a given $q\in (1,\infty)$ and $\omega \in
  A_q$, there is a sequence of measurable, tensor-valued functions $a^k, s^k: \Omega \to \setR^{N\times n}$, $k \in \N$, such that $k$-uniformly
  \begin{equation}
\norm{a^k}_{L^q_\omega(\Omega)}   + \norm{s^k}_{L^{q'}_\omega(\Omega)}   \le C. \label{bit3}
  \end{equation}
  Furthermore, assume that for every bounded sequence
  $\{c^k\}_{k=1}^{\infty}$ in $W^{1,\infty}_{0} (\Omega)$ and for every bounded solenoidal sequence   $\{d^k\}_{k=1}^{\infty}$ in $W^{1,\infty}_{0, \divergence{}} (\Omega)$ such that
  $$
  \nabla c^k \rightharpoonup^* 0 \qquad \textrm{weakly$^*$ in }
  L^{\infty}(\Omega), \qquad   \nabla d^k \rightharpoonup^* 0 \qquad \textrm{weakly$^*$ in }
  L^{\infty}(\Omega)
  $$
one has
  \begin{align}
    \lim_{k\to \infty} \int_{\Omega} s^k \cdot \nabla d^k \, \text{d}x
    &=0, \label{bit4}
    \\
    \lim_{k\to \infty} \int_{\Omega} a^k_i \partial_{x_j} c^k -
    a^k_j \partial_{x_i} c^k \, \text{d}x &=0 &&\textrm{for all }
    i,j=1,\ldots,n.\label{bit5}
  \end{align}
  and that
  \begin{equation}
     \mathrm{tr} (a^k) \quad \text{converges pointwisely almost everywhere in } \Omega.
    \label{bit6}
    \end{equation}
  Then, there exists a (non-relabeled) subsequence $(a^k,b^k)$  and a non-decreasing sequence of measurable subsets
  $\Omega_j\subset\Omega$, with $|\Omega \setminus \Omega_j|\to 0$ as $j\to
  \infty$, such that
  \begin{align}
  a^k &\rightharpoonup a &&\textrm{weakly in } L^1(\Omega), \label{bitfa}\\
  s^k &\rightharpoonup s &&\textrm{weakly in } L^1(\Omega), \label{bitfb}\\
  a^k \cdot s^k \omega &\rightharpoonup a \cdot s\, \omega &&\textrm{weakly in } L^1(\Omega_j) \quad \textrm{ for all } j\in \mathbb{N}. \label{bitf}
  \end{align}
\end{theorem}


\section{A-priori estimates}
\subsection{A-priori estimates for \eqref{eq:model}.}
In this section we introduce a-priori bounds of solutions in $L^q$ spaces, with $q$ below the natural duality exponent $p$.  
 Since we follow the approach of~\cite{BS16} in the non-Newtonian setting,  a-priori estimates in $L^q$ spaces are not sufficient to prove existence. Indeed, in order to apply the existence machinery developed in~\cite{BBS16} further  estimates in weighted $L^p$ spaces are necessary:

\begin{theorem}
\label{pro: unif-weight-est} Let  $\Omega \subset \mathbb{R}^{3}$ be a Lipschiz domain which is open, and bounded, $A$ satisfying \eqref{eq:coerc}--\eqref{eq:mon} and $h\in L^{1}(\Omega)$ be an a.e. positive function. Then there exists $\varepsilon_0 >0$  depending on $\Omega$, $C_1$, $C_2$, $C_3$ and $p$ such that for any $\varepsilon \in (0, \varepsilon_0)$ the following holds: If $f\in L^{p^{\prime}}\left(\Omega\right)\cap L_{\left(Mh\right)^{-\varepsilon}}^{p^{\prime}}$ and $u\in W_{0,\mathrm{div}}^{1,p}\left(\Omega\right)$ is a weak solution for~\eqref{eq:model}

then necessarily  $\nabla u\in L_{\left(Mh\right)^{-\varepsilon}}^{p}\left(\Omega\right)$ and moreover
\[\int_{\Omega}\frac{\left|\nabla u\right|^{p}}{\left(Mh\right)^{\varepsilon}}\mathrm{d}x\le c\left(p,\Omega,C_1,C_2,C_3\right)\int_{\Omega}\frac{\left|f\right|^{p^\prime}+C_{1}\left(p,\Omega\right)}{\left(Mh\right)^{\varepsilon}}\mathrm{d}x.\]
\end{theorem}
\begin{proof}
We set $g:=h\chi_{\Omega}+\delta$; this function will aproximate $h$ but for $g$ we have that $Mg >\delta$ and so $f,\nabla u\in L_{\left(Mg\right)^{-\varepsilon}}^{p}\left(\Omega\right)$ a priori. This fact will be very important in the end of the proof. 
For any $\lambda >0$ we define $\mathcal{O}\left(\lambda\right):=\left\{ x\in\mathbb{R}^{d};Mg\left(x\right)>\lambda\right\} $. This is an open set because of the sub-linearity of the Maximal operator and if $\mathcal{O}(\lambda)=\mathbb{R}^{d}$ then we set $u_{\mathcal{O}\left(\lambda\right)}=0$. But since $\mathcal{O}(\lambda)$ is a proper open set we may construct the relative truncation $u_{\mathcal{O}\left(\lambda\right)}\in W^{1,p}_{0,\divergence}(\Omega)$ by Theorem~\ref{thm:main3} and use it as a test function. We obtain 
\[
\int_{\Omega}A(\varepsilon u)\cdot\varepsilon u_{\mathcal{O}\left(\lambda\right)}\mathrm{d}x=\int_{\Omega}f\cdot\nabla u_{\mathcal{O}\left(\lambda\right)}\mathrm{d}x
\]
hence by \eqref{eq:coerc} and \eqref{eq:lip2} we find
\[
c\int_{\left\{ Mg\le\lambda\right\} }\left|\varepsilon u\right|^{p}-1\mathrm{d}x\leq \int_{\left\{ Mg\le\lambda\right\} }f\cdot\nabla u\mathrm{d}x+\int_{\left\{ Mg>\lambda\right\} }f\cdot\nabla u_{\mathcal{O}\left(\lambda\right)}\mathrm{d}x-\int_{\left\{ Mg>\lambda\right\} }A(\varepsilon u)\cdot\varepsilon u_{\mathcal{O}\left(\lambda\right)}\mathrm{d}x
\]
 which implies (using the elementary poinwise inequality $\left|\varepsilon u\right|\le\left|\nabla u\right|$ and \eqref{eq:bound}) that
 \begin{equation}\label{testing}
 \int_{\left\{ Mg\le\lambda\right\} }\left|\varepsilon u\right|^{p}\mathrm{d}x
 \le
 c\int_{\left\{ Mg\le\lambda\right\} }\left|f\right|\left|\nabla u\right|+1\mathrm{d}x+c\int_{\left\{ Mg>\lambda\right\} }\left|f\right|\left|\nabla u_{\mathcal{O}\left(\lambda\right)}\right|\mathrm{d}x+c\int_{\left\{ Mg>\lambda\right\} }(\left|\nabla u\right|^{p-1}+1)\left|\nabla u_{\mathcal{O}\left(\lambda\right)}\right|\mathrm{d}x.
 \end{equation}
Now let $G\in L^{p^{\prime}}\left(\Omega\right)$. Then by the definition of $u_{\mathcal{O}}$ on $\left\{ Mg>\lambda\right\}$
 \begin{align*}
(I):=\int_{\left\{ Mg >\lambda\right\}}\abs{G}\left|\nabla u_{\mathcal{O}}\right|\mathrm{d}x & \le\sum_{i}\int_{Q_{i}}\left|G\right|\left|\sum_{j\in A_{i}}\nabla\mathrm{curl}\left(\varphi_{j}w_{j}\right)\right|\mathrm{d}x\\
 & \le 2\sum_{i}\sum_{j\in A_{i}}\int_{Q_{i}}\left|G\right|\left|\nabla^{2}\left(\left(w_{j}-w_{i}\right)\varphi_{j}\right)\right|\mathrm{d}x
\end{align*} where on the last inequality we used the partion of unity property. Now for $x\in Q_i$, we find
\begin{align*}
\left|\nabla^{2}\left(\left(w_{j}(x)-w_{i}(x)\right)\varphi_{j}(x)\right)\right| 
& \le
\left|\nabla\left(w_{j}(x)-w_{i}(x)\right)\right|\left|\nabla\varphi_{j}(x)\right|+\left|w_{j}(x)-w_{i}(x)\right|\left|\nabla\varphi_{j}^{2}(x)\right|
\\
 & \le\frac{\left\Vert \nabla\left(w_{j}-w_{i}\right)\right\Vert _{L^{\infty}\left(Q_{i}\cap \frac32Q_j\right)}}{r_{i}}+\frac{\left\Vert w_{j}-w_{i}\right\Vert _{L^{\infty}\left(Q_{i}\cap \frac32Q_j\right)}}{r_{i}^{2}}
\end{align*}
pointwise and if we integrate this on $Q_i$ we have by Lemma~\ref{lem:RT1}, H\"older's inequality and Proposition~\ref{pro:cubes}, (c)
\begin{align*}
 (I)&\le
c\sum_{i}\sum_{j\in A_{i}}\left|Q_{i}\right|\fint_{Q_{i}}|G|\mathrm{d}x  \left(\fint_{\frac{3}{2}Q_{i}}\left|\nabla^{2}w\right|\mathrm{d}x+\fint_{\frac{3}{2}Q_{j}}\left|\nabla^{2}w\right|\mathrm{d}x\right)
\\
&
=c\sum_{i}\sum_{j\in A_{i}}\left|Q_{i}\right|\fint_{Q_{i}}\frac{|G|}{\left(Mg\right)^{\frac{\alpha}{p}}}\left(Mg\right)^{\frac{\alpha}{p}}\mathrm{d}x 
 \left(\fint_{\frac{3}{2}Q_{i}}\frac{\left|\nabla^{2}w\right|}{\left(Mg\right)^{\frac{\alpha}{p^{\prime}}}}\left(Mg\right)^{\frac{\alpha}{p^{\prime}}}\mathrm{d}x+\fint_{\frac{3}{2}Q_{j}}\frac{\left|\nabla^{2}w\right|}{\left(Mg\right)^{\frac{\alpha}{p^{\prime}}}}\left(Mg\right)^{\frac{\alpha}{p^{\prime}}}\mathrm{d}x\right)
\\
&
=c\sum_{i}\sum_{j\in A_{i}}\left|Q_{i}\right|\left(\fint_{Q_{i}}\frac{|G|^{p^{\prime}}}{\left(Mg\right)^{\frac{\alpha p^{\prime}}{p}}}\mathrm{d}x\right)^{\frac{1}{p^{\prime}}} 
 \left(\left(
\fint_{\frac{3}{2}Q_{i}}\frac{\left|\nabla^{2}w\right|^p}{\left(Mg\right)^{\frac{\alpha p}{p^{\prime}}}}\mathrm{d}x\right)^{\frac{1}{p}}
+\left(\fint_{\frac{3}{2}Q_{j}}\frac{\left|\nabla^{2}w\right|^p}{\left(Mg\right)^{\frac{\alpha p}{p^{\prime}}}}\mathrm{d}x\right)^{\frac{1}{p}}\right)
\\
&\quad \times\left(\fint_{5Q_{i}}\left(Mg\right)^{\alpha}\mathrm{d}x\right)
.\end{align*}
We now estimate $\fint_{5Q_{i}}\left(Mg\right)^{\alpha}\mathrm{d}x$. By Lemma ~\ref{weight2} it follows that $\left(Mg\right)^{\alpha}\in\mathcal{A}_{1}$ and hence $M\left(Mg\right)^{\alpha}\le c\left(\alpha\right)\left(Mg\right)^{\alpha}$. Since by the Whitney covering (Proposition~\ref{pro:cubes}, (b)) we find $8Q_{i} \cap \mathcal{O}_{\lambda} ^{c} \neq \emptyset$ it follows that if $x_0$ belongs to this intersection then \[\fint_{9Q_{i}}\left(Mg\right)^{\alpha}\mathrm{d}x\lesssim\left(Mg\right)^{\alpha}\left(x_{0}\right)\le c\left(\alpha\right)\lambda^{\alpha}\] and thus \[
\fint_{5Q_{i}}\left(Mg\right)^{\alpha}\mathrm{d}x\le c\left(\alpha\right)\lambda^{\alpha}
.\]
Observe that by Young's inequality for any $i\in \mathbb{N}$ and any $j \in A_{i}$
\begin{align*}
 & \lambda^{\alpha}\left|Q_{i}\right|\left(\fint_{Q_{i}}\frac{\left|G\right|^{p^{\prime}}}{\left(Mg\right)\frac{\alpha p}{p^{\prime}}}\mathrm{d}x\right)^{\frac{1}{p^{\prime}}}\left(\fint_{\frac{3}{2}Q_{i}}\frac{\left|\nabla^{2}w\right|^{p}}{\left(Mg\right)^{\frac{\alpha p}{p^{\prime}}}}\mathrm{d}x\right)^{\frac{1}{p}}
 \\
&\quad 
\le c\left(\alpha\right)\left|Q_{i}\right|  \left(\fint_{Q_{i}}\frac{\left|G\right|^{p^{\prime}}\lambda^{\frac{\alpha p^{\prime}}{p}}}{\left(Mg\right)\frac{\alpha p^{\prime}}{p}}\mathrm{d}x+\fint_{\frac{3}{2}Q_{i}}\frac{\left|\nabla^{2}w\right|^{p}\lambda^{\frac{\alpha p}{p^\prime}}}{\left(Mg\right)^{\frac{\alpha p}{p^{\prime}}}}\mathrm{d}x\right)
\end{align*} 
which implies by summing over all such $i$ and $j$ that 
\[
\int_{\left\{ Mg>\lambda\right\} }\abs{G}\left|\nabla u_{\mathcal{O}\left(\lambda\right)}\right|\mathrm{d}x\le c\left(\Omega,\alpha\right)\int_{\left\{ Mg>\lambda\right\} \cap\Omega}\frac{\left|G\right|^{p^{\prime}}\lambda^{\frac{\alpha p^{\prime}}{p}}}{\left(Mg\right)^{\frac{\alpha p^{\prime}}{p}}}+\frac{\left|\nabla^{2}w\right|^{p}\lambda^{\frac{\alpha p}{p^{\prime}}}}{\left(Mg\right)^{\frac{\alpha p}{p^{\prime}}}}\mathrm{d}x.
\]
By choosing $G\in\left\{ \left|f\right|+1,\ \left|\nabla u\right|^{p-1}\right\} $ we obtain 
\[\int_{\left\{ Mg>\lambda\right\} }(\left|f\right|+1)\left|\nabla u_{\mathcal{O}\left(\lambda\right)}\right|\mathrm{d}x\le c\left(\Omega,\alpha\right)\int_{\left\{ Mg>\lambda\right\} \cap\Omega}\frac{(1+\left|f\right|^{p^{\prime}})\lambda^{\frac{\alpha p^{\prime}}{p}}}{\left(Mg\right)^{\frac{\alpha p^{\prime}}{p}}}+\frac{\left|\nabla^{2}w\right|^{p}\lambda^{\frac{\alpha p}{p^{\prime}}}}{\left(Mg\right)^{\frac{\alpha p}{p^{\prime}}}}\mathrm{d}x
\]
 and 
\[
\int_{\left\{ Mg>\lambda\right\} }\left|\nabla u\right|^{p-1}\left|\nabla u_{\mathcal{O}\left(\lambda\right)}\right|\mathrm{d}x\le c\left(\Omega,\alpha\right)\int_{\left\{ Mg>\lambda\right\} \cap\Omega}\frac{\left|\nabla u\right|^{p}\lambda^{\frac{\alpha p^{\prime}}{p}}}{\left(Mg\right)^{\frac{\alpha p^{\prime}}{p}}}+\frac{\left|\nabla^{2}w\right|^{p}\lambda^{\frac{\alpha p}{p^{\prime}}}}{\left(Mg\right)^{\frac{\alpha p}{p^{\prime}}}}\mathrm{d}x.\]
If we add them and use (\ref{testing} ) it follows that 
\[
\int_{\left\{ Mg\le\lambda\right\} }\left|\varepsilon u\right|^{p}\mathrm{d}x\le c\int_{\left\{ Mg\le\lambda\right\} }\left|f\right|\left|\nabla u\right|+1\mathrm{d}x+c\left(\Omega,\alpha\right)\int_{\left\{ Mg>\lambda\right\} }\frac{\left(1+\left|f\right|^{p^{\prime}}+\left|\nabla u\right|^{p}\right)\lambda^{\frac{\alpha p^{\prime}}{p}}}{\left(Mg\right)^{\frac{\alpha p^{\prime}}{p}}}+\frac{\left|\nabla^{2}w\right|^{p}\lambda^{\frac{\alpha p}{p^{\prime}}}}{\left(Mg\right)^{\frac{\alpha p}{p^{\prime}}}}\mathrm{d}x.
\]
Let us set $\overline{\left(p-1\right)}:=\min\left(p-1,\ \left(p-1\right)^{-1}\right)$.
We multiply the above inequality by $\lambda^{-1-\varepsilon}$  with $\varepsilon\in\left(0,\ \overline{p-1} \right)$   and integrate  over $\lambda \in  (0, \infty)$ to obtain 
\begin{equation}\label{eps-int} I_{0}\left(\varepsilon\right)\le I_{1}\left(\varepsilon\right)+I_{2}\left(\varepsilon\right)+I_{3}\left(\varepsilon\right)
\end{equation} where 
\begin{align*}
I_{0}\left(\varepsilon\right) & :=\int_{0}^{\infty}\int_{\left\{ Mg\le\lambda\right\} }\left|\varepsilon u\right|^{p}\mathrm{d}x\lambda^{-1-\varepsilon}\mathrm{d}\lambda\\
I_{1}\left(\varepsilon\right) & := c\int_{0}^{\infty}\int_{\left\{ Mg\le\lambda\right\} }\left|f\right|\left|\nabla u\right|+1\mathrm{d}x\lambda^{-1-\varepsilon}\mathrm{d}\lambda\\
I_{2}\left(\varepsilon\right) & :=c\left(\Omega,\alpha\right)\int_{0}^{\infty}\int_{\left\{ Mg>\lambda\right\} }\frac{\left(1+\left|f\right|^{p^{\prime}}+\left|\nabla u\right|^{p}\right)}{\left(Mg\right)^{\frac{\alpha p^{\prime}}{p}}}\mathrm{d}x\lambda^{\frac{\alpha p^{\prime}}{p}-1-\varepsilon}\mathrm{d}\lambda\\
I_{3}\left(\varepsilon\right) & :=c\left(\Omega,\alpha\right)\int_{0}^{\infty}\int_{\left\{ Mg>\lambda\right\} }\frac{\left|\nabla^{2}w\right|^{p}}{\left(Mg\right)^{\frac{\alpha p}{p^{\prime}}}}\lambda^{\frac{\alpha p}{p^{\prime}}-1-\varepsilon}\mathrm{d}x\mathrm{d}\lambda.
\end{align*}
We apply Fubini's theorem several times to obtain 
\begin{align*}
I_{0}\left(\varepsilon\right) & =\int_{\Omega}\int_{Mg\left(x\right)}^{+\infty}\lambda^{-1-\varepsilon}\mathrm{d}\lambda\left|\varepsilon u\right|^{p}\mathrm{d}x\\
 & =\frac{1}{\varepsilon}\int_{\Omega}\frac{\left|\varepsilon u\right|^{p}}{\left(Mg\right)^{\varepsilon}}\mathrm{d}x
\end{align*}
 and 
 \begin{align*}
I_{1}\left(\varepsilon\right) & =\frac{c}{\varepsilon}\int_{\Omega}\frac{1+\left|f\right|\left|\nabla u\right|}{\left(Mg\right)^{\varepsilon}}\mathrm{d}x.
\end{align*}
 Also \[I_{2}\left(\varepsilon\right)	=c\left(\Omega,\alpha\right)\frac{1}{\frac{\alpha p^{\prime}}{p}-\varepsilon}\int_{\Omega}\frac{\left(1+\left|f\right|^{p^{\prime}}+\left|\nabla u\right|^{p}\right)}{\left(Mg\right)^{\varepsilon}}\mathrm{d}x\]
 
 and \[I_{3}\left(\varepsilon\right)	=c\left(\Omega,\alpha\right)\frac{1}{\frac{\alpha p}{p^{\prime}}-\varepsilon}\int_{\Omega}\frac{\left|\nabla^{2}w\right|^{p}}{\left(Mg\right)^{\varepsilon}}\mathrm{d}x.\]
 Therefore (\ref{eps-int}) becomes
 \[
 \int_{\Omega}\frac{\left|\nabla u\right|^{p}}{\left(Mg\right)^{\varepsilon}}\mathrm{d}x
 \le
 c\int_{\Omega}\frac{1+\left|f\right|\left|\nabla u\right|}{\left(Mg\right)^{\varepsilon}}\mathrm{d}x+\frac{\varepsilon c\left(\Omega,\alpha\right)}{\overline{\left(p-1\right)}-\varepsilon}\int_{\Omega}\frac{1+\left|f\right|^{p^{\prime}}+\left|\nabla u\right|^{p}}{\left(Mg\right)^{\varepsilon}}\mathrm{d}x\] after applying the weighted Korn inequality (Theorem~\ref{thm:kornweight}) and the fact that $\varepsilon < \overline{(p-1)}$.
  We point out that $\left(Mg\right)^{-\varepsilon}=\left(Mg\right)^{-\left(p-1\right)\frac{\varepsilon}{p-1}}\in\mathcal{A}_{p}$ by using Lemma ~\ref{lem:tur2} since $\frac{\varepsilon}{p-1}\in\left(0,1\right)$. Further, by applying Young's inequality we obtain 
\[ \int_{\Omega}\frac{\left|\nabla u\right|^{p}}{\left(Mg\right)^{\varepsilon}}\mathrm{d}x\le c\left(p,\varepsilon\right)\int_{\Omega}\frac{1+\left|f\right|^{p^{\prime}}}{\left(Mg\right)^{\varepsilon}}\mathrm{d}x+\frac{\varepsilon c\left(\Omega,\alpha,p\right)}{\overline{\left(p-1\right)}-\varepsilon}\int_{\Omega}\frac{\left|f\right|^{p^{\prime}}+\left|\nabla u\right|^{p}}{\left(Mg\right)^{\varepsilon}}\mathrm{d}x\] and we  notice that the term involving $\left|\nabla u\right|^{p}$ can be absorbed into the left hand side if $\epsilon_0$ is chosen small enough. Actually this is possible, since $\lim\limits_{\varepsilon\to0}\frac{\varepsilon c\left(\Omega,\alpha,p\right)}{\overline{\left(p-1\right)}-\varepsilon}=0$. 
The argument is concluded by taking the limit when $\delta \to 0$ and applying the Monotone Convergence Theorem.

\end{proof}
Please observe that the assumption $u\in W^{1,p}(\Omega)$ is only formal.
\begin{corollary}
\label{cor:1}
Let $p\in (1,\infty)$, $\Omega$ a bounded, open and Lipschitz and $A$ satisfying \eqref{eq:coerc}--\eqref{eq:mon}, then there exists an $\varepsilon >0$  depending on $\Omega$, $C_1$, $C_2$, $C_3$ and $p$ such that if $ q \in [p^{\prime}-\varepsilon, p^{\prime}]$ and $f\in L^{q}\cap L^p\left(\Omega,\mathbb{R}^{3\times3}\right)$ and $\left(u,\pi\right)\in W_{0,\mathrm{div}}^{1,p}\left(\Omega;\mathbb{R}^{3}\right)\times L_{0}^{p}\left(\Omega;\mathbb{R}\right)$ a solution to~\eqref{eq:model}.
Then \[
\int_{\Omega}\left|\nabla u\right|^{q\left(p-1\right)}\mathrm{d}x\le c\left(C_1,C_2,C_3,p,q,\Omega\right)\left(\int_{\Omega}\left|f\right|^{q}\mathrm{d}x+1\right)
\] and 
\[\int_{\Omega}\left|\nabla u\right|^{p}M\left(\left|f\right|+1\right)^{q-p^{\prime}}\mathrm{d}x\le c\left(C_1,C_2,C_3,p,q,\Omega\right)\left(\int_{\Omega}\left|f\right|^{q}\mathrm{d}x+1\right)\]
for some positive constant $c\left(C_1,C_2,C_3,p,q,\Omega\right)>0$ .
\end{corollary}
\begin{proof}
We apply Theorem ~ \ref{pro: unif-weight-est} for  $\varepsilon=p^{\prime}-q\in\left(0,\varepsilon_{0}\right)$ and  $\omega:=M\left(\left|f\right|+1\right)^{-\varepsilon}$
\begin{align*}
\int_{\Omega}\left|\nabla u\right|^{p}\omega\mathrm{d}x & \le c\int_{\Omega}\left(\left|f\right|^{p^{\prime}}+1\right)\omega\mathrm{d}x\\
 & \le c\int_{\Omega}\left|f\right|^{p^{\prime}}\cdot M\left(f\right)^{-\varepsilon}+c_{1}\cdot1\mathrm{d}x\\
 & \le c\int_{\Omega}M\left(f\right)^{q}+c_{1}\mathrm{d}x\\
 & \le c\int_{\Omega}\left|f\right|^{q}+c_{1}\mathrm{d}x
\end{align*} where the last  inequality makes use of the continuity of the maximal operator.
On the other hand we can apply Young's inequality with $\frac{q}{p^{\prime}}+\frac{\varepsilon}{p^{\prime}}=1$ and the continuity of the maximal operator to obtain 
\begin{align}
\label{eq:qtoweight}
\begin{aligned}
\int_{\Omega}\left|\nabla u\right|^{q\left(p-1\right)}\mathrm{d}x & =\int_{\Omega}\frac{\left|\nabla u\right|^{q\left(p-1\right)}}{M\left(\left|f\right|+1\right)^{\frac{\varepsilon q}{p^{\prime}}}}M\left(\left|f\right|+1\right)^{\frac{p^{\prime}}{\varepsilon q}}\mathrm{d}x\\
 & \le\int_{\Omega}\frac{\left|\nabla u\right|^{p}}{M\left(\left|f\right|+1\right)^{\varepsilon}}\mathrm{d}x+\int_{\Omega}M\left(\left|f\right|+1\right)^{q}\mathrm{d}x\\
 & \le c\int_{\Omega}\left|f\right|^{q}\mathrm{d}x+cc_{2}.
 \end{aligned}
\end{align}
\end{proof}
\subsection{A-priori estimates for \eqref{eq:navier-stokes}--the case $p>2$}
\begin{proposition}
\label{pro:p}
Let $p\in (2,\infty)$, $\Omega$ be bounded, open and Lipschitz and $A$ satisfying \eqref{eq:bound}, \eqref{eq:mon} and \eqref{eq:mons}, then there exists an $\varepsilon >0$  depending on $\Omega$ and $p$ such that if $ q \in [p^{\prime}-\varepsilon, p^{\prime}]$ and $f\in L^{q}\cap L^p\left(\Omega,\mathbb{R}^{3\times3}\right)$ and $\left(u,\pi\right)\in W_{0,\mathrm{div}}^{1,p}\left(\Omega;\mathbb{R}^{3}\right)\times L_{0}^{p}\left(\Omega;\mathbb{R}\right)$ a solution to~\eqref{eq:navier-stokes}.
Then \[
\int_{\Omega}\left|\nabla u\right|^{q\left(p-1\right)}+ \left|\nabla u\right|^{p}M\left(\left|f\right|+1\right)^{q-p^{\prime}}\mathrm{d}x
\le c\left(C_1,C_2,C_3,p,q,\Omega\right)\left(\int_{\Omega}\left|f\right|^{q}\mathrm{d}x+1\right)^{\frac{1}{p-2}+\alpha\big(\frac{p'-q}{q}\big)},
\]
for $\alpha$ defined in \eqref{eq:alpha}.
\end{proposition}
\begin{proof}
The basic idea is to split the problem into a Stokes part which is contained in $W^{1,p}(\Omega)$ and to estimate the error. 

Next we solve the following auxiliary problem (which exists due to the assumption that $f\in L^{p'}$)

\begin{align}
\label{NS-b1}
\begin{cases}
-\mathrm{div}A(x,\varepsilon v(x))+\nabla\pi_2(x)=-\mathrm{div}\ f(x) & \mathrm{in}\ \ \Omega\\
\mathrm{div}\ v=0 & \mathrm{in}\ \ \Omega\\
v=0 & \mathrm{on}\  \partial\Omega.
\end{cases}
\end{align}
For $q\in (p'-\epsilon,p')$ we consider
\[
\omega=M(\abs{f}\chi_{\Omega}+1)^{q-p'}.
\]
Corollary ~ \ref{cor:1} implies the existence of a $v$ that satisfies
\begin{align}
\label{eq:vp}
\int_{\Omega} \abs{\nabla v}^{p} \omega\,\text{d}x\leq c\int_{\Omega} \abs{f}^{q}+1 \,\text{d}x
\end{align}
Next we observe, that formally
\begin{align}
\label{eq:differ}
-\mathrm{div}(A(x,\varepsilon u(x))-A(x,\varepsilon v(x)))+\nabla(\pi-\pi_2)(x)=-\mathrm{div}((u(x)\otimes u(x)).
\end{align}
we find by \eqref{eq:mons} that 
\begin{align*}
\int_{\Omega}\left|\varepsilon u-\varepsilon v\right|^{p}\text{dx} & \lesssim\int_{\Omega}\left(A\left(\varepsilon u\right)-A\left(\varepsilon v\right)\right):\varepsilon\left(u-v\right)\text{d}x+1\\
 & =\int_{\Omega}u\otimes u\cdot \nabla \left(u-v\right)\text{d}x +1\\
\end{align*}
%
%
%
This implies by the structure of the convective term, Young's inequality and Korn's inequality (Theorem~\ref{thm:kornweight}) we find for $\delta\in(0,1)$
\begin{align*}
\skp{u\otimes u}{\nabla (u-v)}
&= \skp{u\otimes (u-v)}{\nabla (u-v)}+
\skp{u\otimes v}{\nabla (u-v)} 
\\
&= \int_\Omega u\cdot \nabla \frac{\abs{u-v}^2}{2}\, dx+
\skp{u\otimes v}{\nabla (u-v)} 
\\
&\leq \delta \int_{\Omega}\abs{\varepsilon u-\varepsilon v}^p\, \text{d}x + c_\delta\int_{\Omega} \abs{u\otimes v}^{p'}\, \text{d}x.
\end{align*}
This implies (by absorption) that
\begin{align*}
\int_{\Omega}\abs{\varepsilon u-\varepsilon v}^p\, \text{d}x\lesssim \int_{\Omega} \abs{u\otimes v}^{p'}\, \text{d}x.
\end{align*}
Let $\frac{2p'}{p^*}=1-\gamma>0$. 
We now want to apply Theorem \ref{thm:weight-embedding}. For this, we would need to have $\omega^{\left(3-p\right)/3}\in A_{p^{*}/p^{\prime}+1}$. Here we need to notice that the weight $\omega$ is defined via the Hardy-Littlewood maximal function and using Lemma \ref{lem:tur2} it is enough to check if \[\left(q-p^{\prime}\right)\frac{3-p}{3}=-\alpha\frac{p^{*}}{p^{\prime}}\quad\text{for some}\ \alpha\in\left(0,1\right).\]
But by a simple computation we obtain \[\alpha=\left(p^{\prime}-q\right)\frac{3-p}{3}\cdot\frac{p^{\prime}}{p^{*}}=\frac{p^{\prime}-q}{p^{\prime}-1}\in\left(0,1\right).\]
Consequently it follows that 
\begin{equation}\label{eq:embedding}
\left(\int_{\Omega}\left|u\right|^{p^{*}}\omega^{3/(3-p)}\text{d}x\right)^{1/p^{*}}\le c\left(p,A_{p^{*}/p^{\prime}+1}\left(\omega^{\left(3-p\right)/3}\right)\right)\left(\int_{\Omega}\left|\nabla u\right|^{p}\omega \text{d}x\right)^{1/p}.
\end{equation} 
Now we define $b=p$ and so  $b^*=p^*=\frac{3p}{3-p}$ for $p<3$ and $b=2$ and so  $b^*=2^*=6>2p'$ for $p\geq 3$ and $\gamma$ by $\frac{2p'}{b^*}+\gamma=1$. Please observe that this is possible since for $p<3$
\[
\frac{2p'}{p^*}=\frac{2p(3-p)}{3p(p-1)}<1\text{ (which is possible for all $p>\frac{9}{5}$)}.
\] 
By H\"older's inequality and Theorem~\ref{thm:Tur} we obtain \footnote{Please observe that we may assume in case $p>3$ that $p'-q$ is small enough such that $\omega\in A_2$.} 
\begin{align}
\label{eq:forNS}
\begin{aligned}
 \int_{\Omega} \abs{u\otimes v}^{p'} \,\text{d}x&=  \int_{\Omega} \abs{u\otimes v}^{p'} \frac{\omega^\frac{2p'3}{b^*(3-b)}}{\omega^\frac{2p'3}{b^*(3-b)}}\,\text{d}x
 \\
 &\lesssim  \bigg(\int_{\Omega} \abs{u}^{b^*} \omega^\frac{3}{3-b}\,\text{d}x\bigg)^\frac{p'}{b^*}\bigg(\int_{\Omega} \abs{v}^{b^*} \omega^\frac{3}{3-b}\,\text{d}x\bigg)^\frac{p'}{b^*}\bigg(\int_{\Omega} \omega^\frac{-6p'}{(3-b)b^*\gamma}\, \text{d}x\bigg)^{\gamma}
 \\
 &\lesssim  \bigg(\int_{\Omega} \abs{\nabla u}^{b} \omega\,\text{d}x\bigg)^\frac{p'}{b}\bigg(\int_{\Omega} \abs{\nabla v}^{b} \omega\,\text{d}x\bigg)^\frac{p'}{b}\bigg(\int_{\Omega} \omega^\frac{-6p'}{3b\gamma}\, \text{d}x\bigg)^{\gamma}
 \\
 &\lesssim  \bigg(\int_{\Omega} \abs{\nabla u}^{p} \omega\,\text{d}x\bigg)^\frac{p'}{p}\bigg(\int_{\Omega} \abs{\nabla v}^{p} \omega\,\text{d}x\bigg)^\frac{p'}{p}\bigg(\int_{\Omega} \omega^\frac{-2p'}{b\gamma}\, \text{d}x\bigg)^{\gamma}
 \end{aligned}
\end{align}

Since $p>2$ we assume that $p'-q$ is small enough, such that $\frac{2(p'-q)p'}{b\gamma}\leq q$, in which case we find that 
\[
\bigg(\int_{\Omega} \omega^\frac{-2p'}{b\gamma}\, \text{d}x\bigg)^{\gamma}\leq C\bigg(\int_{\Omega} (M(\abs{f}\chi_{\Omega}+1))^\frac{2(p'-q)p'}{b\gamma}\, \text{d}x\bigg)^{\gamma}
\leq C\bigg(\int_{\Omega} \abs{f}^q+1\text{d}x\bigg)^\frac{2(p'-q)p'}{b q},
\]
with $C$ depending on $p, q$ and $\Omega$ but independent of $f$.
Hence (for $p>2$) we find by~\eqref{eq:vp} and Young's inequality that 
\begin{align*}
\int_{\Omega} \abs{u\otimes v}^{p'} \,\text{d}x&\leq \delta \int_{\Omega} \abs{\nabla u}^{p} \omega\,\text{d}x +c_\delta  \bigg(\int_{\Omega}\abs{f}^q+1\, \text{d}x\bigg)^{\big(\frac{p'}{p}+\frac{2(p'-q)p'}{b q}\big) \frac{p}{p-p'}}
\\
&\leq \delta \int_{\Omega} \abs{\nabla u}^{p} \omega\,\text{d}x +c_\delta   \bigg(\int_{\Omega}\abs{f}^q+1\, \text{d}x\bigg)^{\frac{1}{p-2}+\frac{(p'-q)}{q} \frac{2p}{b(p-2)}}
\end{align*}
And so, by Korn's inequality, Young's inequality Theorem~\ref{thm:kornweight}, \eqref{eq:alpha} and \eqref{eq:vp}
\begin{align*}
\int_{\Omega}\abs{\nabla u}^p\omega\, \text{d}x&\lesssim 
\int_{\Omega}\abs{\varepsilon u-\varepsilon v}^p\, \text{d}x+ \int_{\Omega}\abs{\nabla v}^p\omega\, \text{d}x
\\
&\lesssim \int_{\Omega} \abs{u\otimes v}^{p'}\, \text{d}x+\int_{\Omega}\abs{f}^q+1\, \text{d}x
\\
&\lesssim  \delta \int_{\Omega} \abs{\nabla u}^{p} \omega\,\text{d}x +c_\delta   \bigg(\int_{\Omega}\abs{f}^q+1\, \text{d}x\bigg)^{\frac{1}{p-2}+\alpha((p'-q)/q)}
\end{align*}
This implies the uniform bound by absorption. Finally the $W^{1,q(p-1)}(\Omega)$-bound follows as in \eqref{eq:qtoweight}.
%
\end{proof}
\subsection{A-priori estimates for \eqref{eq:navier-stokes}--the case $p=2$}
\begin{proposition}
\label{pro:2}
Let $p=2$, $\Omega$ a bounded, open and $C^1$ and $A$ satisfying \eqref{eq:mon}--\eqref{(c)} (with $p=2$) and \eqref{eq:uhlenbeck}, then for $ q \in [\frac{12}{7}, 2)$ and $f\in L^{q}\cap L^2\left(\Omega,\mathbb{R}^{3\times3}\right)$ and $\left(u,\pi\right)\in W_{0,\mathrm{div}}^{1,2}\left(\Omega;\mathbb{R}^{3}\right)\times L_{0}^{2}\left(\Omega;\mathbb{R}\right)$ a solution to~\eqref{eq:navier-stokes}.
Then \[
\int_{\Omega}\left|\nabla u\right|^{q}+\left|\nabla u\right|^{2}M\left(\left|f\right|+1\right)^{q-2}\mathrm{d}x\le C
\]
with $C'$ depending on $\norm{f}_{L^q(\Omega)},C_1,C_2,C_3$ and the linear at infinity condition. 
\end{proposition}
\begin{proof}
The basic idea is to split the problem into a large part which is contained in $W^{1,p}(\Omega)$ and a small very-weak part. Let $\tilde{\delta}\in (0,\frac{\nu}{4})$. By the assumption \eqref{eq:uhlenbeck} there is a $K>2$, such that 
\begin{align*}
\abs{\nu z- A(x,z)}+\abs{\Ibb\nu-D_zA(x,z)}\leq \tilde{\delta}\text{ for all }\abs{z}\geq K\text{ and all }x\in \Omega. 
\end{align*}
We define $\phi\in C^2([0,\infty),[0,1])$, such that $\chi_{[K/2,\infty)}\leq \phi \leq \chi_{[K,\infty)})$ and $\phi'\leq 1$
\begin{align*}
\tilde{A}(x, z)= \nu z+ \phi(\abs{z})(A(x,z)-\nu z)
\end{align*}
Please observe that $\tilde{A}$ satisfies \eqref{eq:bound} (with $p=2$) and \eqref{eq:uhlenbeck}. Moreover it satisfies \eqref{eq:mons} for $p=2$ and $C_3=0$, since in case $\abs{z_1}\geq \abs{z_2}$
\begin{align*}
(\tilde{A}(x,z_1)-\tilde{A}(x,z_2))\cdot(z_1-z_2)&= \nu\abs{z_1-z_2}^2 +\big((\phi(\abs{z_1})(A(x,z_1)-\nu z_1)-\phi(\abs{z_2})(A(x,z_2)-\nu z_2)\big)\cdot (z_1-z_2)
\\
&= \nu\abs{z_1-z_2}^2  + \phi(\abs{z_2}) \big(A(x,z_1)-\nu z_1-A(x,z_2)+\nu z_2\Big)\cdot(z_1-z_2) 
\\
&\quad + \big(\phi(\abs{z_1})-\phi(\abs{z_2})\big)(A(x,z_1)-\nu z_1)\cdot(z_1-z_2)
\\
&=:\nu\abs{z_1-z_2}^2+(I)+(II)
\end{align*}
Due to the support of $\phi$ and the fact that $\varphi\left(\left|u\right|\right)-\varphi\left(\left|v\right|\right)\le1$ for all $u,v  \in \mathbb{R}$ , we find that
\begin{align*}
\left|\left(II\right)\right| & \le\tilde{\delta}\left|\varphi\left(\left|z_{1}\right|\right)-\varphi\left(\left|z_{2}\right|\right)\right|\left|z_{1}-z_{2}\right|\\
 & \le\frac{\nu}{4}\left|z_{1}-z_{2}\right|^{2}
\end{align*}

Similarly we find 
\begin{align*}
\abs{(I)}&=\phi(\abs{z_2})\absBB{\sum_{i,j=1}^{3}\int_{z_1^{ij}}^{z_2^{ij}}\partial_{z^{ij}}A(x,\xi)-\delta_{ij}\nu\,d \xi (z_1^{ij}-z_2^{ij})}
\leq \tilde{\delta}\phi(\abs{z_2})\abs{z_1-z_2}^2\leq \frac{\nu}{4}\abs{z_1-z_2}^2
\end{align*}
And so
\begin{align*}
(\tilde{A}(x,z_1)-\tilde{A}(x,z_2))\cdot(z_1-z_2)\geq \nu\abs{z_1-z_2}^2-\abs{(I)}-\abs{(II)}\geq \frac{\nu}{2}\abs{z_1-z_2}^2.
\end{align*}
We split 
\[
f=g_k+b_k:=f\chi_{\abs{f}\leq k}+f\chi_{\abs{f}>k}.
\]
Next we solve the following auxiliary Stokes problem: 

\begin{align}
\label{NS-b}
\begin{cases}
-\divergence (\tilde{A}(\varepsilon v))+\nabla\pi_2(x)=-\mathrm{div}\ b_k(x) & \mathrm{in}\ \ \Omega\\
\mathrm{div}\ v=0 & \mathrm{in}\ \ \Omega\\
v=0 & \mathrm{on}\  \partial\Omega.
\end{cases}
\end{align}
Since $f\in L^2(\Omega)$ the existence follows by monotone operator theory.

Moreover, we find that in $\Omega$ (as $\Delta v=2\divergence(\epsilon v)$)
\begin{align}
\label{NS-new}
-\frac{\nu}{2}\Delta v+\nabla\pi_2(x)=-\mathrm{div}(b_k(x)+\nu \epsilon v-\tilde{A}(x,\varepsilon v))=-\mathrm{div}\big( b_k(x)+\phi(\abs{\varepsilon v})(A(x,\varepsilon v)-\nu \varepsilon v)\big)
\end{align}
We consider
\[
\omega=M(\abs{f}\chi_{\Omega}+1)^{q-2}\in A_2\text{ as }q-2\in (-1,0).
\]
Now \cite[Lemma~3.2]{BBS16} and \eqref{eq:uhlenbeck} in combination with the support of $\phi$ implies that
\begin{align*}
\int_\Omega \abs{\nabla v}^2\omega \text{d}x &\leq c\int_\Omega \abs{b_k}^q\, \text{d}x
+c\int_\Omega \abs{\phi(\abs{\varepsilon v})(A(x,\varepsilon v)-\nu \varepsilon v)}^2\omega \text{d}x
\\
&\leq  c\int_\Omega \abs{b_k}^q\, \text{d}x
+c\tilde{\delta} \int_\Omega \abs{\varepsilon v}^2\omega \text{d}x
\end{align*}
using that $\supp(\phi)\subset [K,\infty)$. Choosing $\tilde{\delta}$ small enough implies
\[
\int_\Omega \abs{\nabla v}^2\omega \text{d}x\leq c\int_\Omega \abs{b_k}^q\, \text{d}x.
\]
By Fubini theorem
\[
\norm{f}_{L^q_\omega(\Omega)}^q\sim \int_{l=0}^\infty \omega(\{\abs{f}>l\})l^q\, dl,
\]
which implies that for every $\beta>0$ there exists a $k$, such that
\[
\norm{b_k}_{L^q_\omega(\Omega)}^q=\int_{k}^\infty \omega(\{\abs{f}>l\})l^q\, dl= \beta.
\]
 Hence we find
\begin{align}
\label{eq:beta}
\int_{\Omega} \abs{\nabla v}^{2} \omega\,\text{d}x\leq c\int_{\Omega} \abs{b_k}^{q} \,\text{d}x\leq c\beta.
\end{align}
Next we observe, that  that
\begin{align}
\label{eq:differ1}
-\frac{\nu}{2}\Delta(u-v)+\nabla(\pi-\pi_2)(x)=-\mathrm{div}(\nu (\varepsilon (u-v)) -(A(\cdot, \varepsilon u)-\tilde{A}(\cdot, \varepsilon v))+(u(x)\otimes u(x))+g_k(x)).
\end{align}

By testing \eqref{eq:differ1} with $u-v$ and using Young's inequality

we find for $\delta>0$ the estimate (using $p\geq 2$)
\begin{align}
\label{eq:key}
\frac{\nu}{2}\int_{\Omega} \abs{\nabla u-\nabla v}^2\, \text{d}x
&= \skp{u\otimes u}{\nabla (u-v)} + \skp{\nu (\varepsilon (u-v)) -(A(\cdot, \varepsilon u)-\tilde{A}(\cdot, \varepsilon v))+g_k }{\nabla(u-v)}
\end{align}
This implies by the structure of the convective term and the symmetry of the convective term that
\begin{align*}
\skp{u\otimes u}{\nabla (u-v)}
&= \skp{u\otimes v}{\nabla (u-v)} 
\\
&\leq \delta \int_{\Omega}\abs{\nabla u-\nabla v}^2\, \text{d}x + c_\delta\int_{\Omega} \abs{u\otimes v}^{2}\, \text{d}x
\end{align*}
We estimate further
\[
\skp{g_k}{\nabla (u-v)}\leq \delta \int_{\Omega}\abs{\nabla (u-v)}^2\, \text{d}x+ c_\delta \int_{\Omega} \abs{g_k}^{2}\, \text{d}x.
\]
And finally
\begin{align*}
\skp{\nu (\varepsilon (u-v)) -(A(\cdot, \varepsilon u)-\tilde{A}(\cdot, \varepsilon v))}{\nabla (u-v)}&\leq \delta \int_{\Omega}\abs{\nabla (u-v)}^2\, \text{d}x
\\
&+ c_\delta \int_{\Omega} \abs{\nu (\varepsilon (u-v)) -(A(\cdot, \varepsilon u)-\tilde{A}(\cdot, \varepsilon v))}^{2}\, \text{d}x
\end{align*}
But now
\begin{align*}
\abs{\nu (\varepsilon (u-v)) -(A(\cdot, \varepsilon u)-\tilde{A}(\cdot, \varepsilon v))}
&\leq cK + \abs{ A(\cdot,\varepsilon u)-\nu \varepsilon u }\chi_{\{\abs{\varepsilon v}\leq 2K\}}\chi_{\{\abs{\varepsilon u}\geq 4K\}}
\\
&\quad  + \abs{ A(\cdot,\varepsilon u)-A(\cdot, \varepsilon v)-\nu\varepsilon (u-v)}\chi_{\{\abs{\varepsilon v}\geq 2K\}}\chi_{\{\abs{\varepsilon u}\geq 4K\}}
\\
&\leq cK 
+\tilde{\delta}\abs{\varepsilon u}\chi_{\{\abs{\varepsilon v}\leq 2K\}}\chi_{\{\abs{\varepsilon u}\geq 4K\}}
\\
& \quad + \absBB{\sum_{i,j}^3 \int_{(\varepsilon u)^{i,j}}^{(\varepsilon v)^{i,j}}\partial_{z^{ij}} A(\cdot,\xi)-\delta_{ij}\nu\, d\xi}\chi_{\{\abs{\varepsilon v}\geq 2K\}}\chi_{\{\abs{\varepsilon u}\geq 4K\}}
\\
&\leq cK 
+2\tilde{\delta}\abs{\varepsilon (u-v)}
+\sum_{i,j}^3 \int_{[(\varepsilon u)^{i,j},(\varepsilon v)^{i,j}]}\tilde{\delta}  d\xi
\\
&\leq cK + 3\tilde{\delta}\abs{\varepsilon (u-v)}.
\end{align*}
This implies (by choosing $\delta=\frac{1}{6}$ and $\tilde{\delta}\leq \frac{\delta}{3c_\delta}$ and absorption) that
\begin{align*}
\int_{\Omega}\abs{\varepsilon u-\varepsilon v}^2\, \text{d}x\lesssim \int_{\Omega} \abs{u\otimes v}^{2}+{g_k}^2\, \text{d}x+K\leq  \int_{\Omega} \abs{u\otimes v}^{2}+k^{2-q}{f}^q\, \text{d}x+K.
\end{align*}
In order to estimate the convective term we use \eqref{eq:forNS} and take $p=2$, $p^*=6$ and $\gamma=\frac13$, which implies that
\[
\frac{-2p'}{p\gamma}=\frac{-2}{\frac13}=-6
\]
and so
\begin{align*}
\begin{aligned}
 \int_{\Omega} \abs{u\otimes v}^{2} \,\text{d}x&\lesssim  \int_{\Omega} \abs{\nabla u}^{2} \omega\,\text{d}x\int_{\Omega} \abs{\nabla v}^{p} \omega\,\text{d}x\bigg(\int_{\Omega} \omega^{-6}\, \text{d}x\bigg)^\frac1{3}
 \end{aligned}
\end{align*}
Since
\[
\omega^{-6}=
M(\abs{f}\chi_\Omega+1)^{(2-q)6},
\]
and
\[
(2-q)6\leq q\text{ by the assumption that }q\in [\frac{12}{7},2],
\]
 we find that
\begin{align*}
\begin{aligned}
 \int_{\Omega} \abs{u\otimes v}^{2} \,\text{d}x&\lesssim \beta\bigg(\int_{\Omega} \abs{f}^q+1\, \text{d}x\bigg)^\frac1{3} \int_{\Omega} \abs{\nabla u}^{2} \omega\,\text{d}x.
 \end{aligned}
\end{align*}
Hence we can estimate
\begin{align*}
\int_{\Omega}\abs{\nabla u}^2\omega\, \text{d}x&\lesssim 
\int_{\Omega}\abs{\nabla( u-  v}^2\, \text{d}x+ \int_{\Omega}\abs{\nabla v}^2\omega\, \text{d}x
\\
&\lesssim \int_{\Omega} \abs{u\otimes v}^{2}+k^{2-q}\abs{f}^q\, \text{d}x+\beta +K
\\
&\lesssim  \int_{\Omega} \abs{\nabla u}^{2} \omega\,\text{d}x\beta\bigg(\int_{\Omega} \abs{f}^q+1\, \text{d}x\bigg)^\frac1{3} + C(\norm{f}_{L^q(\Omega)})
\end{align*}
By choosing $\beta$ small enough (meaning $k$ large enough) we can absorb and find
\begin{align*}
\int_{\Omega}\abs{\nabla u}^p\omega\, \text{d}x&\leq C(\norm{f}_{L^q(\Omega)}),
\end{align*}
and so (using H\"older's inequality as in Corollary~\ref{cor:1}) we find
\begin{align*}
\int_{\Omega}\abs{\nabla u}^q\, \text{d}x&\leq C(\norm{f}_{L^q(\Omega)}).
\end{align*}

%
\end{proof}

\section{Existence of very-weak solutions}
\noindent
In this section we prove that there exists a very weak solution to our problem. Actually, we will prove  Theorem~\ref{thm:main2}, Theorem~\ref{thm:p-NAVIER-STOKES} and Theorem~\ref{thm:main3} simultaneously since once the a-priori estimates are established the limit passage procedure  is the same. The proof is achieved in the following three subsections.

\subsection{The approximating system}
The existence follows for all three cases in the same way. The idea will be to consider a sequence of approximate solutions and then pass to the limit accordingly. Let us consider $f\in L^{q}\left(\Omega\right)$ and $f_{k}:=\min\left\{ k,\left|f\right|\right\} f/\left|f\right|$. Then $f_{k}\in L^{\infty}\left(\Omega\right)\cap L^{q}\left(\Omega\right)$ with $\left|f_{k}\right|\nearrow\left|f\right|$ and (via Monotone Convergence) \begin{equation}\label{eq:mon-conv}
f_{k}\to f\ \mathrm{in}\ L^{q}\left(\Omega\right).
\end{equation}

Further we define 
Now, for each $k$, there exists $(u_k,\pi_k)\in W_{0,\mathrm{div}}^{1,p}\left(\Omega\right)\times L^p_0(\Omega)$ which is a solution to \eqref{eq:navier-stokes}  or \eqref{eq:model} with right hand side $f_k$.
This is due to the fact that $f_{k}\in L^{\infty}\left(\Omega\right)\subset L^{p^{\prime}}\left(\Omega\right)$.  The proof of this last fact is classical. 
 We set $\varepsilon=p^{\prime}-q\in\left(0,\varepsilon_{0}\right)$ and  $\omega:=M\left(\left|f\right|+1\right)^{-\varepsilon}$. Then these solutions satisfy uniform estimates by applying Corollary~\ref{cor:1}, Proposition~\ref{pro:p} or Proposition~\ref{pro:p}:
\begin{align*}
\int_{\Omega}\left|\nabla u_{k}\right|^{p}\omega\mathrm{d}x +\int_{\Omega}\left|\nabla u_{k}\right|^{q\left(p-1\right)}\mathrm{d}x \leq C
\end{align*}
 Using these estimates and the reflexivity of the spaces 
\begin{align}
u_{k}\rightharpoonup u &  & \mathrm{weakly}\ \mathrm{in}\ W_{0,\mathrm{div}}^{1,q\left(p-1\right)}\left(\Omega\right) \label{eq:weakconv1}\\
\nabla u_{k}\rightharpoonup\nabla u &  & \mathrm{weakly}\ \mathrm{in}\ L_{\omega}^{p}\left(\Omega\right)\cap L^{q\left(p-1\right)}\left(\Omega\right)\label{eq:weakconv2}\\
A(\cdot,\varepsilon u_{k})\rightharpoonup\overline{S} &  & \mathrm{weakly}\ \mathrm{in}\ L_{\omega}^{p^{\prime}}\left(\Omega\right)\cap L^{q}\left(\Omega\right)\label{eq:weakconv3}.
\end{align}
and if we pass to the limit as $k \to \infty$ 
we end up with the following a priori estimate:
 \begin{equation}\label{eq:a-priori}
\int_{\Omega}\left|\nabla u\right|^{p}\abs{\overline{A}}^{p'}\omega\mathrm{d}x+\int_{\Omega}\left|\nabla u\right|^{q\left(p-1\right)} +\abs{\overline{A}}^{q}\mathrm{d}x\le C.
\end{equation}
Let us prove now that $u$ is a weak solution. Notice that we can pass to the limit  (in case of \eqref{eq:model})  using
 (\ref{eq:mon-conv}) and (\ref{eq:weakconv3})to find
 \begin{equation}\label{eq:S-bar}
\int_{\Omega}\overline{A}\cdot\nabla\varphi\mathrm{d}x=\int_{\Omega}f\cdot\nabla\varphi\mathrm{d}x\ \text{ for all }\varphi\in W_{0,\mathrm{div}}^{1,p}\left(\Omega\right), 
 \end{equation}
 and in case of $\eqref{eq:navier-stokes}$ we use the fact that $ W^{1,q(p-1)}(\Omega)$ compactly embeds into $L^2(\Omega)$ and find
  \begin{equation}\label{eq:S-bar2}
\int_{\Omega}(\overline{A}-(u\otimes u))\cdot\nabla\varphi\mathrm{d}x=\int_{\Omega}f\cdot\nabla\varphi\mathrm{d}x\ \text{ for all }\varphi\in W_{0,\mathrm{div}}^{1,p}\left(\Omega\right), 
 \end{equation}

\subsection{Establishing the non-linearity} 
In this subsection we aim to show that
 \begin{equation}\label{eq:S-bar-2}
 \overline{A}=A(\cdot,\varepsilon u). 
 \end{equation} 
 This will be achieved by using Theorem ~\ref{thm:swbdcl}. Indeed, we choose $a^{k}:=\nabla u_{k}$  $s^{k}:=A(\cdot,\varepsilon u_{k})$ , $q:=p$ and $n=N=3$ and $\omega$ as before. By applying $( \ref{eq:a-priori})$ we can see that 
 \[
\left\Vert a^{k}\right\Vert _{L_{\omega}^{p}\left(\Omega\right)}+\left\Vert s^{k}\right\Vert _{L_{\omega}^{p^{\prime}}\left(\Omega\right)}\le c\left\Vert \nabla u_{k}\right\Vert _{L_{\omega}^{p}\left(\Omega\right)}
\]
which means that ~$(\ref{bit3})$ is fulfilled. Then we have that with $g_k=f_k$ in case of $\eqref{eq:model}$ and $g_k=f_k+u_k\otimes u_k$ in case of~\eqref{eq:navier-stokes}
\[
\lim_{k\to\infty}\int_{\Omega}s^{k}\cdot\nabla d^{k}\mathrm{d}x=\lim_{k\to\infty}\int_{\Omega}g_{k}\cdot\nabla d^{k}\mathrm{d}x=0
\]
 using the equation
 $(\ref{eq:mon-conv})$ and the hypothesis on $d^{k}$; this implies $(\ref{bit4})$. Last but not least $(\ref{bit5})$ and $(\ref{bit6})$ follow by the fact that $a^{k}$ is a gradient.  So we apply Theorem ~\ref{thm:swbdcl} to get a sequence of measurable sets $\Omega_{j} \subset \Omega$ with $\abs{\Omega \setminus \Omega_{j} }\to 0$ as $j\to \infty$ so that 
 \begin{equation}\label{eq:Minty1}
A(\cdot, \varepsilon u_{k})\cdot\nabla u_{k}\omega\rightharpoonup\overline{A}\cdot\nabla u\omega\ \mathrm{weakly\ \mathrm{in}}\ L^{1}\left(\Omega_{j}\right).
\end{equation}\label{eq:Minty2}
Now notice that for any $B\in L_{\omega}^{p}\left(\Omega\right)$ we obtain  \begin{equation}
\left(A(\cdot, \varepsilon u_{k})-A(\cdot, B^s)\right)\cdot\left(\nabla u_{k}-B\right)\omega\rightharpoonup\left(\overline{A}-A(\cdot, B^s)\right)\cdot\left(\nabla u-B\right)\omega\ \mathrm{weakly\ \mathrm{in}}\ L^{1}\left(\Omega_{j}\right)
\end{equation}
where we denoted $B^s := \frac{B+B^{T}}{2}$. Denote $A(Q):=\abs{Q}^{p-2}Q$ for $Q:\Omega \to \mathbb{R}^{3\times3}$. Then $\left(A\left(Q^{s}\right)-A\left(P^{s}\right)\right)\cdot\left(Q-P\right)\ge0$. Thus \[\int_{\Omega_{j}}\left(\overline{A}-A(\cdot, B^s)\right)\cdot\left(\nabla u-B\right)\omega\mathrm{d}x\ge0\] which can be rewritten  as \[\infty >
\int_{\Omega}\left(\overline{A}-A(\cdot, B^s)\right)\cdot\left(\nabla u-B\right)\omega\mathrm{d}x\ge\int_{\Omega\setminus\Omega_{j}}\left(\overline{A}-A(\cdot, B^s)\right)\cdot\left(\nabla u-B\right)\omega\mathrm{d}x.
\] Since $\left|\Omega\setminus\Omega_{j}\right|\to0$ as $j\to \infty$ we can apply the dominated donvergence theorem and obtain that 
\[\int_{\Omega}\left(\overline{A}-A(\cdot, B^s)\right)\cdot\left(\nabla u-B\right)\omega\mathrm{d}x\ge0.\]
We can now choose $B:=\nabla u-\delta G$ with $G\in L^{\infty}\left(\Omega\right)$ and $\delta >0$. The last relation becomes \[
\int_{\Omega}\left(\overline{A}-A\left(\cdot,\varepsilon u-\delta G^{s}\right)\right)\cdot G\omega\mathrm{d}x\ge0
.\]
We let $\delta \to 0+$ and we use again Dominated Convergence Theorem to obtain \[
\int_{\Omega}\left(\overline{A}-A\left(\cdot,\varepsilon u\right)\right)\cdot G\omega\mathrm{d}x\ge0\ \text{ for all } G\in L^{\infty}\left(\Omega\right).
\]
Finally, we choose \[G:=-\frac{\overline{A}-A\left(\varepsilon u\right)}{\left|\overline{A}-A\left(\varepsilon u\right)\right|+1}\] and we conclude that $(\ref{eq:S-bar-2})$ is true. \\
The proof is concluded, once the existence and the estimates for the pressure $\pi$ are shown:

\subsection{Existence \& estimates for pressure}\label{sec:pressure}

 We start by noticing that the following holds for $g=f$ in case of $\eqref{eq:model}$ and $g=f+u\otimes u$ in case of~\eqref{eq:navier-stokes}. Since in case of \eqref{eq:navier-stokes} we have that $2q\leq 4$ and so $L^{2q}\Omega\subset W^{-1,q(p-1)}(\Omega)$ 
 We find that
 \[
 \int_{\Omega}A(\cdot,\varepsilon u){\cdot}\varepsilon\varphi\ \mathrm{d}x=\int_{\Omega}g{\cdot}\nabla\varphi\ \mathrm{d}x\quad\text{ for all }\varphi\in W_{0,\mathrm{div}}^{1,q^{\prime}}\left(\Omega\right).
\]
 The weak formulation for our system of equations can also be  rewritten as 
 \[
\int_{\Omega}A(\cdot,\varepsilon u){\cdot}\varepsilon\varphi\ \mathrm{d}x-\int_{\Omega}\pi\mathrm{div}\varphi \mathrm{d}x=\int_{\Omega}g{\cdot}\nabla\varphi\ \mathrm{d}x\quad\text{ for all }\varphi\in W_{0}^{1,q^{\prime}}\left(\Omega\right).
\] or equivalently  
\[
\int_{\Omega}\pi\mathrm{div}\varphi\mathrm{d}x=\int_{\Omega}A(\cdot,\varepsilon u){\cdot}\varepsilon\varphi\ \mathrm{d}x-\int_{\Omega}g{\cdot}\nabla\varphi\ \mathrm{d}x\quad\text{ for all }\varphi\in W_{0}^{1,q^{\prime}}\left(\Omega\right).
\]
We consider the following mapping $ \mathcal{F}:
L_{0}^{q^{\prime}}\left(\Omega\right)\longmapsto\mathbb{R}$ given by 
\[
a\longmapsto\mathcal{F}(a):= \int_{\Omega}A(\cdot,\varepsilon u){\cdot}\varepsilon\mathrm{Bog}\left(a\right)\ \mathrm{d}x-\int_{\Omega}g{\cdot}\nabla\mathrm{Bog}\left(a\right)\ \mathrm{d}x.
\]
It is standard to apply H\"{o}lder's inequality and to conclude that $\mathcal{F}$ is linear and continuous. Therefore there exists \[\pi\in L_{0}^{q}\left(\Omega\right)\cong\left(L_{0}^{q^{\prime}}\left(\Omega\right)\right)^{\prime}\] such that 
\[
\int_{\Omega}\pi a\mathrm{d}x=\int_{\Omega}A(\cdot,\varepsilon u){\cdot}\varepsilon\mathrm{Bog}\left(a\right)\ \mathrm{d}x-\int_{\Omega}g{\cdot}\nabla\mathrm{Bog}\left(a\right)\ \mathrm{d}x
.\]
Since $\varphi\in W_{0}^{1,q^{\prime}}\left(\Omega\right)\implies\mathrm{div}\varphi\in L_{0}^{q^{\prime}}\left(\Omega\right)$ and since $\mathrm{div}\varphi=\mathrm{div}\mathrm{Bog}\mathrm{div}\varphi$  we  obtain :
\begin{align*}
\int_{\Omega}\pi\mathrm{div}\varphi\mathrm{d}x & =\int_{\Omega}A(\cdot,\varepsilon u){\cdot}\varepsilon\mathrm{Bog}\left(\mathrm{div}\varphi\right)\ \mathrm{d}x-\int_{\Omega}g{\cdot}\nabla\mathrm{Bog}\left(\mathrm{div}\varphi\right)\ \mathrm{d}x\\
 & =\int_{\Omega}A(\cdot,\varepsilon u){\cdot}\varepsilon\varphi\ \mathrm{d}x-\int_{\Omega}g{\cdot}\nabla\varphi \text{d}x\quad\text{ for all }\varphi\in W_{0}^{1,q^{\prime}}\left(\Omega\right)
\end{align*}
where the last equality is due to the fact that $\mathrm{\varphi-Bog}\mathrm{div}\varphi\in W_{0,\mathrm{div}}^{1,q^{\prime}}\left(\Omega\right)$ can be used as a test function. \\
By duality, we have 
\[
\left\Vert \pi\right\Vert _{L^{q}\left(\Omega\right)}=\left\Vert \mathcal{F}\right\Vert _{op}\le\left\Vert A(\cdot,\varepsilon u)\right\Vert _{L^{q}\left(\Omega\right)}+\left\Vert g\right\Vert _{L^{q}\left(\Omega\right)}
\]
And in exactly the same manner we obtain the estimate $\left\Vert \pi\right\Vert _{L_{\omega}^{p'}\left(\Omega\right)}\le c\norm{ g}_{L^{p^{\prime}}_\omega\left(\Omega\right)}+c\norm{\abs{\nabla v}^{p-1}} _{L^{p^{\prime}}_\omega\left(\Omega\right)}$  using now the fact \[\left(L_{\omega}^{p'}\left(\Omega\right)\right)^{\prime}\cong L_{\omega^{\prime}}^{p}\left(\Omega\right),\quad\omega^{\prime}:=\omega^{-\frac{1}{p-1}}.\] 
 This now  ends the proof of the a-priori bounds of the Theorem~\ref{thm:main2} and Theorem~\ref{thm:p-NAVIER-STOKES} and Theorem~\ref{thm:new-main}.

\appendix
\section{Construction of the relative truncation}
\label{appendix}
\subsection{Inverse curl operator in weighted spaces}
 We will provide the following theorem of the inverse curl operator for weighted spaces. It seems to be an improvement to Lipschitz domains even in Lebesgue spaces~\cite[Corollary 2.3]{BS90}.
 \begin{theorem}\label{thm:weight-curl}
Let $p\in (1,\infty)$, let $\Omega \subset \mathbb{R}^{3}$ be a bounded subset of $\Rbb^3$ with Lipschitz boundary 
 and let $u \in  W_{0, \text{div}}^{1,p}(\Omega)$. 
 There exists $w\in W_{0}^{2,p}(\Omega)$ such that $\text{curl}\left(w\right)=u$. In particular
 \[
 \int_\Omega \abs{\nabla^2 w}^p\, \text{d}x\leq c\int_\Omega \abs{\nabla u}^p\, \text{d}x,
 \]
 where $c$ depends on $p,n$ and the domain.
 
 Moroever, if $\omega\in \Acal_{p}$  and   $\nabla u \in  L^p_\omega (\Omega)$,
 then there exists $w\in W_{0}^{2,1}(\Omega)$ such that $\text{curl}\left(w\right)=u$ and $\nabla^2 w\in L^p_\omega(\Omega)$. In particular
 \[
 \int_\Omega \abs{\nabla^2 w}^p\, d\omega \leq c\int_\Omega \abs{\nabla u}^p\, d\omega,
 \]
 where $c$ depends on $p,n$, the $\Acal_p$-constant and the domain.
 \end{theorem}

 The proof relies on the following extension results \cite[Theorem~2.1.13]{Turesson 2000}.
 \begin{theorem}
 \label{thm:ext}
 Let $p\in (1,\infty)$, let $\Omega \subset \mathbb{R}^{3}$ be an open, bounded and Lipschitz domain. Let $\omega\in \Acal_p$. Assume that $\nabla^k g\in L^p_{\omega}(\Omega^c)$, then the there exists an extension $E(g)\in W^{k,p}_{\omega}(\Omega)$, such that
 \[
 E(g)=g\text{ in }\Omega^c\text{ and }\norm{E(g)}_{ W^{k,p}_{\omega}(\Omega)}\leq c\norm{\nabla^k g}_{L^p_{\omega}(\Omega^c)}.
 \]
 \end{theorem}

 \begin{proof}[Proof of Theorem~\ref{thm:weight-curl}] 
 We begin with the same strategy as in~\cite{BS90}. Let first $u\in C^{\infty}_{0,\divergence}(\Omega)$ (the general results follows then by density). We extend $u$ by $0$ to the whole space and take the global solution of the inverse curl of $u$
\[
\tilde{w}(y)=\int_{\Rbb^3}\frac{\curl(u)(z)}{\abs{y-z}}\, dz.
\] 
The mapping $\curl(u)\mapsto \nabla^2\tilde w$ is a singular integral operator and hence~(cf. \cite{stein})
 \[
\norm{\nabla^2\tilde w}_{L^{p}_{\omega}(\Rbb^3)}\leq c\norm{\nabla u}_{L^p_{\omega}(\Omega)}.
\]
It is easy to check that $\curl\tilde w=u\chi_{\Omega}$. However $\nabla \tilde w\neq 0$ on $\partial\Omega$. We will correct the boundary value with another singular operator on $\Omega$. First, since $curl\tilde w=0$ on $\Omega^c$ \textit{the Helmholtz decomposition} implies that there is a $z\in W^{3,1}_{\textrm{loc}}(\Rbb^3\setminus \Omega)$ satisfying $\nabla z(x)=\tilde{w}(x)$ for all $x\in \Omega^c$.
Obviously
\[
\norm{\nabla^3 z}_{L^p_\omega(\Rbb^3\setminus\Omega)}\leq \norm{\nabla^2\tilde w}_{L^{p}_{\omega}(\Rbb^3)}\leq c\norm{u}_{W^{1,p}_\omega(\Omega)}.
\] 
Since $\Omega$ is bounded there exists $R>0$ such that $\Omega \subseteq B_R(0)$. We consider a smooth function $\eta$ that equals $1$ in $B_R(0)$, is $0$ outside $2B_R(0)$ and $0 \le \eta \le 1$. Thus $\eta z \in W^{3,p}_{\omega} (\Omega^c)$. Now we apply the extension operator from \ref{thm:ext} so that $E(\eta z) \in W^{3,p}_\omega(\mathbb{R}^3)$ and 
\[
\norm{E(\eta z)}_{W^{3,p}_{\omega}(2B_R(0))}\leq c\norm{u}_{W^{1,p}_\omega(\Omega)}.
\]
 We now set $w:=\tilde{w}-\nabla E(\eta p)$. Consequently $w=\nabla w=0$ on $\partial \Omega$. Furthermore, $\curl(w)=\curl(\tilde{w})=u\chi_\Omega$ and  $w\in W_{0,\omega}^{2,p}(\Omega)$ by the above estimates. 
\end{proof}
\subsection{Construction of the relative truncation}
 In the following section we construct a \textit{relative truncation} $u_\mathcal{O}$, where $\mathcal{O}$ is the set over which we wish to change the function $u$. Th truncation is supposed to preserve solenoidality {\bf and} the  zero boundary values. We follow the approach in \cite{BS16}. If not specified otherwise we use the trivial extension by $0$ of all functions to the whole-space without any further notice.

The construction makes use of the celebrated Whitney covering.
We use here a version which is present in ~\cite{Grf14} and then slightly modified in ~\cite{BS16} and ~\cite{BDS13}.

\begin{proposition} 
\label{pro:cubes}
 Let $\mathcal{O}$ be an open and proper subset of $\mathbb{R}^{d}$ . Then there exists a countable family $Q_{i}$ of closed, dyadic cubes such that:
\begin{enumerate}[(a)]
\item $\bigcup_{i}Q_{i}=\mathcal{O}$ and all the cubes $Q_{i}$ have disjoint interiors.
\item $\text{diam}(Q_{i})<\text{dist}\left(Q_{i},\mathcal{O}^{c}\right)\le4\text{diam}\left(Q_{i}\right)$.
\item If $Q_{i} \cap Q_{j} \neq \emptyset$, then $\frac{\text{\text{diam}(\ensuremath{Q_{i}})}}{\text{diam}(Q_{j})}\in\left[\frac{1}{2},2\right]$
\item For given $Q_{i}$, there exists at most $4^{d}-2^{d}$ cubes $Q_{j}$ touching $Q_{i}$ (boundaries intersect, but not the interiors).
\item The family of cubes  $\left\{ \frac{3}{2}Q_{i}\right\} _{i\in\mathbb{N}}$ has finite intersection. The family can be split in $4^{d}-2^{d}$ disjoint families.
\item There is a partition of unity , $\psi_{i}\in C^{\infty}_{c}\left(\mathbb{R}^{d}\right)$, such that $\chi_{\frac{1}{2}Q_{i}}\le\psi_{i}\le\chi_{\frac{9}{8}Q_{i}}$ and $\text{diam}(Q_{i})\left|\nabla\psi_{i}\right|\le c\left(d\right)$ uniformly.

\end{enumerate}
\end{proposition}
\begin{proof} We shall first fix the notation; namely, for any $m \in \mathbb{Z}$ we denote by: \begin{itemize}
\item $\mathcal{D}_{m}$ the set of all dyadic cubes of length $2^{-m}$
\item $\mathcal{O}_{m}:=\left\{ x\in\mathcal{O} :\  2\sqrt{d}2^{-m}<\text{dist}\left(x,\mathcal{O}^{c}\right)\le4\sqrt{d}2^{-m}\right\} $
\item $\mathcal{F}_{m}:=\left\{ Q\in\mathcal{D}_{m}:Q\cap\mathcal{O}_{m}\neq\emptyset\right\} $
\item $\mathcal{F}^{\prime}:=\bigcup_{m\in\mathbb{Z}}\mathcal{F}_{m}$
\end{itemize}
It is immediate to see that $\bigcup_{m\in \mathbb{Z}} \mathcal{O}_{m}=\mathcal{O}$ and , consequently $\bigcup_{Q\in \mathcal{F}^{\prime}} Q=\mathcal{O}$.
The property (b) holds for any $Q\in \mathcal{F}^{\prime}$ since if $Q\in \mathcal{F}_{m}$ and $x \in Q \cap  \mathcal{O}_{m}$, then 
$$\text{diam}(Q)<\text{dist}\left(x,\mathcal{O}^{c}\right)-\text{diam}(Q)\le\text{dist}\left(Q,\mathcal{O}^{c}\right)\le\text{dist}\left(Q,\mathcal{O}^{c}\right)\le4\text{diam}(Q).$$
In order to fulfill the condition a) we need to choose a subfamily of $\mathcal{F}^{\prime}$  such that any two cubes in this new subfamily have disjoint interiors. Since two dyadic cubes have disjoint interiors or one contains the other we can define, for any $Q\in \mathcal{F}^{\prime}$, $Q^{\text{max}}$ to be the maximal cube of $\mathcal{F}^{\prime}$ that contains $Q$. Now set $\mathcal{F}:=\left\{ Q^{\text{max}}:\ Q\in\mathcal{F}^{\prime}\right\} $. Then any two cubes in $\mathcal{F}$ have disjoint interiors by maximality. Now the conditions a) and b) are fulfilled. \\
To prove c), consider $Q, Q^{\prime} \in \mathcal{F}$ with $Q \cap Q^{\prime} \neq \emptyset$. We have $$\text{diam}\left(Q\right)<\text{dist}\left(Q,\mathcal{O}^{c}\right)\le\text{\text{dist}\ensuremath{\left(Q,Q^{\prime}\right)}+dist}\left(Q^{\prime},\mathcal{O}^{c}\right)\le4\text{diam}\left(Q^{\prime}\right)$$
If $l(Q)=2^{-k}$ and $l(Q^{\prime})=2^{-l}$ we have that $-k<-l+2$ or $-k \le -l +1$ and thus $\frac{\text{\text{diam}(Q)}}{\text{diam}(Q^{\prime})}\in\left[\frac{1}{2},2\right]$. \\
Now, following c) we notice that given $Q \in \mathcal{D}_{m}$ from $\mathcal{F}$, any cube from $\mathcal{F}$  that touches Q contains at least one cube in $\mathcal{D}_{m+1}$ that touches $Q$. Thus the number of neighbours of $Q$ is at least $4^{d}-2^{d}$ which is the number of cubes in $\mathcal{D}_{m+1}$ that touch $Q$.  \\
Using the condition b) we can show that $\frac{3}{2}Q_{i} \subset \mathcal{O}$; otherwise there would exist $x \in \frac{3}{2}Q_{i} \cap \mathcal{O}^{c}$ and therefore $$\text{diam}\left(Q_{i}\right)<\text{dist}\left(Q_{i},\mathcal{O}^{c}\right)\le\text{dist}\left(Q_{i},x\right)\le\frac{1}{2}\text{diam}\left(Q_{i}\right)$$ which is a contradiction.  Thus $\bigcup_{i}\frac{3}{2}Q_{i}=\mathcal{O}=\bigcup_{i}Q_{i}$. We see that $\frac{3}{2}Q_{i}$ only intersects its neighbours-by c)- and therefore each $x\in Q_{i}$ is covered by at most $4^{d}-2^{d}$ elements of $\left\{ \frac{3}{2}Q_{k}\right\} _{k\in\mathbb{N}}$. \\
By c) we also see that $\frac{9}{8}Q_{i}$ and $\frac{1}{2}Q_{j}$ do not intersect if $i \neq j$.  Therefore consider $\psi$  a  smooth function such that it  equals  1 on $[-1/2 , 1/2]^{d}=:Q$ and 0 outside $[-9/8, 9/8]^{d} =: Q^{*}$. Then we can define for any $k \in \mathbb{N}$, $\tilde{\psi}_k:=\psi ((x-c_{k} )/ l(Q_{k}))$, where $c_{k}$ is the center of the cube $Q_{k}$ and $l(Q_k)$ its side-length. Finally we consider $\psi_{k}:=\frac{\tilde{\psi}_{k}}{\sigma}$, where $\sigma:=\sum_{k}\tilde{\psi}_{k}$.  Since  $\psi$ is a smooth function with compact support, it is uniformly bounded  (and the same applies for  any $\partial_{i} \psi$)  . This is the partition of unity we wanted. 
\end{proof}

\noindent

Let $\Omega \subset \mathbb{R}^{3}$ be an open, bounded and Lipschitz domain.
We consider $ W_{0, \text{div}}^{1,p}(\Omega)$, the closure in $W^{1,p}(\Omega)$ of the set \[\left\{ \varphi|\ \varphi\in C_{c}^{\infty}\left(\Omega,\mathbb{R}^{3}\right),\ \text{div}\left(\varphi\right)=0\right\} .\]
%

We consider  a function $u \in W_{0}^{1,p}(\Omega)$  and a domain $\Omega \subset \mathbb{R}^{3}$ which is  open, bounded,  with Lipschitz boundary. Then according to Theorem \ref{thm:weight-curl} there is a function $w\in W_{0}^{2,p}\left(\Omega\right)$ for which $\mathrm{curl}w=u$.  Without further notice we extend $w$ by zero outside $\Omega$. \\
Given $\mathcal{O}$ an open and proper set we can consider a Whitney covering and a related partition of unity as in the Proposition ~\ref{pro:cubes}. Then we define
\[w_{\mathcal{O}}:=\begin{cases}
w\left(x\right) & x\in\mathbb{R}^{3}\setminus\mathcal{O}\\
\sum_{i}\varphi_{i}w_{i} & x\in\mathcal{O}
\end{cases}\]
where \[w_{i}:=\begin{cases}
\left(\nabla w\right)_{\frac{3}{2}Q_{i}}\left(x-x_{i}\right)+\left(w-\left(\nabla w\right)_{\frac{3}{2}Q_{i}}\left(x-x_{i}\right)\right)_{\frac{3}{2}Q_{i}} & \mathrm{if}\ \frac{3}{2}Q_{i}\subset\Omega\\
0 & \mathrm{else}
\end{cases} \]
and \[
\left|\varphi_{i}\right|+r_{i}\left|\nabla\varphi_{i}\right|+r_{i}^{2}\left|\nabla^{2}\varphi_{i}\right|\le c\quad\mathrm{where}\ r_{i}:=\mathrm{diam}\left(Q_{i}\right).
\]
Finally we define the set of neighbors of $Q_i$ as $A_i:\{j:\{\phi_j(x)>0:x\in Q_i\}\neq \emptyset\}$. Such that
\[
w_{\mathcal{O}}(x)=\sum_{j\in A_{i}}\phi_j w_j\text{ for  $x\in Q_i$.}
\] 
\subsection{Properties of the relative truncation}

 In this section we prove the following:

\begin{theorem}[Relative truncation]\label{thm:main3}
Let $\mathcal{O}$ a nonempty, open and proper subset of $\mathbb{R}^{3}$. Let $u \in W_{0, \text{div}}^{1,p}(\Omega)$ where $\Omega$ is an open, bounded and Lipschitz subset of $\mathbb{R}^{3}$. Let $\omega\in \mathcal{A}_p$ be a Muckenhoupt weight.  Then there exists a function which we denote $u_{\mathcal{O}} \in W_{0,\text{div}}^{1,p}(\Omega)$ with the following properties:
\begin{equation}\label{eq: lip1}
u_{\mathcal{O}}=u\quad\text{on}\ \mathbb{R}^{3}\setminus\mathcal{O}
\end{equation}
\begin{equation}\label{eq:lip2}
\int_{\Omega}\left|\nabla\left(u-u_{\mathcal{O}}\right)\right|^{p}\text{d}x\le c\left(p\right)\int_{\Omega}\left|\nabla u\right|^{p}\text{d}x
\end{equation}
\begin{equation}\label{eq:lip3}
\int_{\Omega}\left|\nabla\left(u-u_{\mathcal{O}}\right)\right|^{p}\omega\text{d}x\le c\left(p,A_{p}\left(\omega\right)\right)\int_{\Omega}\left|\nabla u\right|^{p}\omega\text{d}x
\end{equation}
and for $q<p$ and $\omega\in A_q$ we find
 \begin{align}
 \label{eq:lip4}
\int_{\Omega}\left|\nabla\left(u-u_{\mathcal{O}}\right)\right|^{q}\omega \text{d}x & \leq c\left(q,A_{q}\left(\omega\right)\right)\omega(\mathcal{O})^{\frac{p-q}p} \bigg(\int_{\Omega}\left|\nabla u\right|^{p}\omega\text{d}x\bigg)^\frac{q}{p}.
\end{align}
\end{theorem}

The theorem is a consequence of the following  lemma.
\begin{lemma} 
\label{lem:RT1}
The following relations hold for $\omega\in A_p$
\begin{enumerate}[(a)]
\item\label{(a)} \[
\int_{\frac{3}{2}Q_{i}}\left|\frac{w-w_{i}}{r_{i}^{2}}\right|^{p}\omega\mathrm{d}x+\int_{\frac{3}{2}Q_{i}}\left|\frac{\nabla\left(w-w_{i}\right)}{r_{i}}\right|^{p}\omega\mathrm{d}x\leq c\left(p,A_p\right)\int_{\frac{3}{2}Q_{i}}\left|\nabla^{2}w\right|^{p}\omega\mathrm{d}x.
\]
\item\label{(b)} If $\abs{Q_{i} \cap Q_j }\neq 0$ then
\[
\int_{Q_{i}\cap \frac{3}{2} Q_j}\!\!\!\!\!\!\!\!\!\!\!\!\frac{\left|w_{j}-w_{i}\right|^{p}}{r_{j}^{2p}}\omega\mathrm{d}x+\int_{Q_{i}\cap \frac{3}{2} Q_j}\!\!\!\!\!\!\!\!\!\!\!\!\frac{\left|\nabla\left(w_{j}-w_{i}\right)\right|^{p}}{r_{j}^{p}}\omega\mathrm{d}x\le c\left(p,A_p\right)\int_{\frac{3}{2}Q_{i}}\left|\nabla^{2}w\right|^{p}\omega\mathrm{d}x+c\left(p,A_p\right)\int_{\frac{3}{2}Q_{i}}\left|\nabla^{2}w\right|^{p}\omega\mathrm{d}x.
\]
\item\label{(c)}  If $\abs{Q_{i} \cap Q_j }\neq 0$   then
\[
\frac{\left\Vert w_{j}-w_{i}\right\Vert _{L^{\infty}\left(Q_{i}\cap \frac{3}{2} Q_j\right)}}{r_{i}^{2}}+\frac{\left\Vert \nabla\left(w_{j}-w_{i}\right)\right\Vert _{L^{\infty}\left(Q_{i}\cap \frac{3}{2} Q_j\right)}}{r_{i}}\le c\fint_{\frac{3}{2}Q_{i}}\left|\nabla^{2}w\right|\mathrm{d}x+c\fint_{\frac{3}{2}Q_{j}}\left|\nabla^{2}w\right|\mathrm{d}x.
\]
\end{enumerate}
\end{lemma}
\begin{proof}
\begin{enumerate}[(a)]
\item We just apply weighted Poincare's inequality~\eqref{eq:weightponc} twice to obtain \[
\int_{\frac{3}{2}Q_{i}}\left|\frac{w-w_{i}}{r_{i}^{2}}\right|^{p}\omega\mathrm{d}x\le c\left(p\right)\int_{\frac{3}{2}Q_{i}}\left|\frac{\nabla\left(w-w_{i}\right)}{r_{i}}\right|^{p}\omega\mathrm{d}x\le c\left(p,A_p\right)\int_{\frac{3}{2}Q_{i}}\left|\nabla^{2}w\right|^{p}\omega\mathrm{d}x.
\]
\item Follows by  applying  the triangle's inequality and \eqref{(a)}.

\item Notice that \[\left\Vert w_{j}-w_{i}\right\Vert _{L^{\infty}\left(Q_{i}\cap \frac32Q_j\right)}\sim\fint_{Q_{i}\cap \frac32Q_{j}}\left|w_{j}-w_{i}\right|^{p}\mathrm{d}x\]  due to the equivalence of the norms $L^\infty$ and $L^{p}$ on the finite dimensional space of linear polynomials. A similar estimate holds for the second term. We add  them   and apply  \eqref{(b)} .
\end{enumerate}
\end{proof}

\begin{proof}[Proof of Theorem~\ref{thm:main3}]
We define $u_\mathcal{O}:=\mathrm{curl}w_\mathcal{O}$. To conclude the proof of Theorem ~\ref{thm:main3} we have to check that the properties \eqref{eq: lip1}-\eqref{eq:lip3} are fulfilled. To this end notice that \eqref{eq: lip1} is immediate as the relation $\mathrm{div}u_\mathcal{O}=0$ follows by construction. The zero boundary values are due to the fact that $w_{\mathcal{O}}\equiv 0$ outside $\Omega$.
To prove \eqref{eq:lip2} we estimate
\begin{align*}
\int_{\Omega}\left|\nabla\left(u-u_{\mathcal{O}}\right)\right|^{p}\mathrm{d}x & =\int_{\Omega\cap\mathcal{O}}\left|\nabla\mathrm{curl}\left(w-w_{\mathcal{O}}\right)\right|^{p}\mathrm{d}x\\
 & =\sum_{i}\int_{Q_{i}}\left|\nabla\mathrm{curl}\left(w-\sum_{j\in A_{i}}w_{j}\varphi_{j}\right)\right|^{p}\mathrm{d}x\\
 & =\sum_{i}\int_{Q_{i}}\left|\nabla\mathrm{curl}\left(\sum_{j\in A_{i}}\left(w-w_{j}\right)\varphi_{j}\right)\right|^{p}\mathrm{d}x\\
 & \le c\left(p\right)\sum_{i}\int_{Q_{i}}\left|\nabla^{2}\left(\sum_{j\in A_{i}}\left(w-w_{j}\right)\varphi_{j}\right)\right|^{p}\mathrm{d}x\\
 & \le c\left(p\right)\sum_{i}\sum_{j\in A_{i}}\int_{Q_{i}\cap \frac{3}{2}Q_j}\left|\nabla^{2}\left(\left(w-w_{j}\right)\varphi_{j}\right)\right|^{p}\mathrm{d}x.
\end{align*}
Further, we estimate pointwisely 
\begin{align}
\label{eq:pointwise}
\begin{aligned}
\left|\nabla^{2}\left(\left(w-w_{j}\right)\varphi_{j}\right)\right|^{p} & \le\left(\left|\nabla^{2}w\right|\cdot\left|\varphi_{j}\right|+\left|\nabla\left(w-w_{j}\right)\right|\cdot\left|\nabla\varphi_{j}\right|+\left|w-w_{j}\right|\cdot\left|\nabla^{2}\varphi_{j}\right|\right)^{p}
\\
 & \le c\left(p\right)\left(\left|\nabla^{2}w\right|^{p}+\frac{\left|\nabla\left(w-w_{j}\right)\right|^{p}}{r_{j}^{p}}+\frac{\left|w-w_{j}\right|^{p}}{r_{j}^{2p}}\right).
 \end{aligned}
\end{align}
Then we find by the previous lemma to obtain 
 \begin{align*}
\int_{\Omega}\left|\nabla\left(u-u_{\mathcal{O}}\right)\right|^{p}\mathrm{d}x & \le\sum_{i}\sum_{j\in A_{i}}c\left(p\right)\left(\int_{\frac{3}{2}Q_{j}}\left|\nabla^{2}w\right|^{p}\mathrm{d}x+\int_{\frac{3}{2}Q_{j}}\left|\nabla^{2}w\right|^{p}\mathrm{d}x\right)\\
 & \le c\left(p\right)\int_{\mathcal{O}\cap\Omega}\left|\nabla^{2}w\right|^{p}\mathrm{d}x\\
 & \le c\left(p\right)\int_{\Omega}\left|\nabla u\right|^{p}\mathrm{d}x.
\end{align*}
By using the weighted estimates we find in the same manner 
 \begin{align*}
\int_{\Omega}\left|\nabla\left(u-u_{\mathcal{O}}\right)\right|^{p}\omega\text{d}x & \le c\left(p\right)\int_{\mathcal{O}\cap\Omega}\left|\nabla^{2}w\right|^{p}\omega\text{d}x\\
 & \le c\left(p,A_{p}\left(\omega\right)\right)\int_{\Omega}\left|\nabla u\right|^{p}\omega\text{d}x
\end{align*}
which  concludes \eqref{eq:lip3} for \eqref{eq:lip4} we take $q<p$ and find by the previous and H\"older's inequality that
 \begin{align*}
\int_{\Omega}\left|\nabla\left(u-u_{\mathcal{O}}\right)\right|^{q}\omega\text{d}x & \lesssim \int_{\mathcal{O}\cap\Omega}\left|\nabla^{2}w\right|^{q}\omega\text{d}x
\\
&\leq \omega(\mathcal{O})^{\frac{p-q}{p}} c\left(q\right)\bigg(\int_{\mathcal{O}\cap\Omega}\left|\nabla^{2}w\right|^{p}\omega\text{d}x\bigg)^\frac{q}{p}
\end{align*}
which ends the proof.
\end{proof}

\subsection{Proof of Theorem~\ref{truncation}}
\label{proof:LT}
We will follow the approach from ~\cite[p. 27-28]{BDS13} .
For any  $\lambda >0$ we set
  $$\left\{M\left(\nabla^{2}w\chi_{\Omega}\right)>\lambda\right\} =:\mathcal{O}_{\lambda}$$
where $M$ is the Hardy - Littlewood maximal operator.
The set $\mathcal{O}_{\lambda}$ is the so-called "bad" set, where the possible  singularities of the function $w$ are contained. We define $w_\lambda:=w_{\mathcal{O_\lambda}}$ defined via Theorem~\ref{thm:main3}
and the {\bf solenoidal Lipschitz truncation of $u$} as 
\[
u_{\lambda}:=\text{curl}\left(w_{\lambda}\right).
\]
We have that $\text{div}u_{\lambda}=\text{div}\text{curl}\left(w_{\lambda}\right)=0$ in $\Omega$. According to Theorem-\ref{thm:main3} we have that $\nabla w_{\lambda}=0$ on $\partial\Omega$ , so $u_{\lambda}=0$ on $\partial\Omega$. 

For $j \in \mathbb{N}$ we find that (by Proposition~\ref{pro:cubes}) $Q_{j} \subset 16 Q_{j}$ and $16Q_{j} \cap \mathcal{O}_{\lambda}^{c} \neq \emptyset$. It follows that $\fint_{16Q_{j}}\left|\nabla^{2}w\right|\mathrm{d}x\le\lambda$ and then  $\fint_{\frac32Q_{j}}\left|\nabla^{2}w\right|\mathrm{d}x\le c\lambda$.
Hence for $x\in \mathcal{O}_\lambda$ using the assumptions on $\psi_j$ and Lemma~\ref{lem:RT1} that
\[
\abs{\nabla^2w(x)}\leq \bigg|\sum_{j\in A_i}\nabla^2(\psi_j(w_j-w_i))\bigg|\leq c\lambda.
\]
Hence
$$
\left|\nabla u_{\lambda}\right|\le 2\abs{\nabla^2 w_\lambda}\le c\lambda.
$$
Now due to the weak $L^1$ estimate for Maximal functions and the $L^p$ bounds of $\nabla^2 w$ we find that
\[
\abs{\mathcal{O}}\lesssim \lambda^{-p}\int_{\Omega}\abs{\nabla u}^p\, \mathrm{d}x.
\] 
Hence the gradient bounds in $L^p_\omega$ and in $L^q$ follow directly from Theorem~\ref{thm:main3}.

Finally we would like to prove that $\left\Vert u_{\lambda}\right\Vert _{L^{q}\left(\Omega\right)}\le c\left\Vert u\right\Vert _{L^{q}\left(\Omega\right)}$ , if $u\in L^{q}\left(\Omega\right)$.
In order to prove this, notice that the only relevant situation is to prove the estimates under $L^{q}\left(\mathcal{O}_{\lambda}\cap\Omega\right)$ . Now because $\sum _{i} \varphi_{i}=1$ it is immediate to check that $\displaystyle u_{\lambda}-u=\sum_{j\in\mathbb{N}}\text{curl}\left(\varphi_{j}\left(w_{j}-w\right)\right)$ and then 
\begin{align*}
\left|u_{\lambda}-u\right| & \le\sum_{j}\left|\text{curl}\left(\varphi_{j}\left(w_{j}-w\right)\right)\right|\lesssim\sum_{j}\left|\nabla\left(\varphi_{j}\left(w_{j}-w\right)\right)\right|\\
 & \lesssim\sum_{j}\left(\left|\nabla\varphi_{j}\right|\left|w_{j}-w\right|+\left|\varphi_{j}\right|\left|\nabla\left(w_{j}-w\right)\right|\right)\\
 & \lesssim\sum_{j}\left(\frac{\left|w_{j}-w\right|}{r_{j}}+\left|\nabla\left(w_{j}-w\right)\right|\right)
\end{align*}
If we integrate the last inequality we will obtain that 
\begin{align*}
\left\Vert u_{\lambda}-u\right\Vert _{L^{q}\left(\mathcal{O}_{\lambda}\cap\Omega\right)}^{q} & \lesssim\sum_{j}\int_{Q_{j}\cap\Omega}\frac{\left|w_{j}-w\right|}{r_{j}}+\left|\nabla\left(w_{j}-w\right)\right|\mathrm{d}x\\
 & \lesssim\sum_{j}\int_{Q_{j}\cap\Omega}\left|\nabla w\right|^{q}\mathrm{d}x
\lesssim\left\Vert u\right\Vert _{L^{q}(\Omega)}^{q};
\end{align*} so $\left\Vert  u_{\lambda}\right\Vert _{L^{q}\left(\Omega\right)}\le c\left\Vert  u\right\Vert _{L^{q}\left(\Omega\right)}$.

\subsection*{Acknowledgments}
C.\ M.\ and  S.\ S.\ thank the support of the research program PRIMUS/19/SCI/01 of Charles University and the support of the program GJ19-11707Y of the Czech national grant agency (GA\v{C}R). They also thank the support of the grant SVV-2019 -260455. S.\ S.\ thanks the the support of the university centre UNCE/SCI/023 of Charles University.

{\small

\bibliographystyle{plain}


\end{document}